\documentclass[a4paper,reqno,11pt]{amsart} 
\usepackage[T1]{fontenc}
\usepackage[utf8]{inputenc}
\usepackage{lmodern}

\usepackage{amsmath, amssymb, graphicx, enumerate, enumitem, color}
\usepackage[usenames,dvipsnames,svgnames,table]{xcolor}
\usepackage[foot]{amsaddr}

\usepackage{multirow}
\usepackage{longtable}
\usepackage{booktabs}
\usepackage{comment}

\makeatletter
\def\thm@space@setup{
  \thm@preskip=4mm
  \thm@postskip=0mm
}
\makeatother

\colorlet{mylinkcolor}{violet}
\colorlet{mycitecolor}{YellowOrange}
\colorlet{myurlcolor}{Aquamarine}

\usepackage{amsthm}
\usepackage{hyperref}
\usepackage{xcolor}
\hypersetup{
  pdfauthor={Heather S. Blake, Piotr Micek, William T. Trotter},
  pdftitle={Boolean dimension and dim-boundedness: Planar cover graph with a zero},
  colorlinks,
  linkcolor={red!50!black},
  citecolor={blue!50!black},
  urlcolor={blue!80!black}
}
\usepackage{cleveref}
\usepackage{mathtools}
\usepackage{needspace}

\DeclarePairedDelimiter\set{\{}{\}}

\DeclarePairedDelimiter\floor{\lfloor}{\rfloor}
\let\preceq\preccurlyeq

\theoremstyle{plain} 
\newtheorem{theorem}{Theorem}
\newtheorem{corollary}[theorem]{Corollary} 
\newtheorem{problem}{Problem}

\newtheorem{conjecture}[problem]{Conjecture}
\newtheorem{lemma}[theorem]{Lemma}
\newtheorem*{completionlemma}{Completion lemma}

\newtheorem{proposition}[theorem]{Proposition}
 
\newtheorem{example}[theorem]{Example}

\theoremstyle{remark}

\theoremstyle{remark}
\newtheorem{claim}{Claim}[theorem]

\bibstyle{plain}

\newcommand{\bbN}{\mathbb{N}}

\newcommand{\cgB}{\mathcal{B}} 
\newcommand{\cgC}{\mathcal{C}}
\newcommand{\cgD}{\mathcal{D}}

\newcommand{\Oh}{\mathcal{O}}

\newcommand{\Inc}{\operatorname{Inc}}

\newcommand{\shad}{\operatorname{shad}}
\newcommand{\sd}{\operatorname{sd}}
\newcommand{\se}{\operatorname{se}}

\newcommand{\bdim}{\operatorname{bdim}}

\newcommand{\NTL}{\operatorname{NTL}} 
\newcommand{\NTR}{\operatorname{NTR}}

\newcommand{\NES}{Ne\v{s}et\v{r}il}
\newcommand{\NESPUD}{Ne\v{s}et\v{r}il and Pudl\'{a}k}

\let\le\leqslant
\let\ge\geqslant
\let\leq\leqslant
\let\geq\geqslant

\let\subsetneq\varsubsetneq

\let\epsilon\varepsilon

\let\setminus\smallsetminus

\newcommand{\Wleft}{W_L}
\newcommand{\Wright}{W_R}

\renewenvironment{enumerate}{\begin{enumorig}[label=\textup{(\roman*)}, noitemsep, 
topsep=2pt plus 2pt, labelindent=.2em, leftmargin=*, widest=iii]}{\end{enumorig}}

\newenvironment{enumerateAlpha}{\begin{enumorig}[label=\textup{(\alph*)}, noitemsep, 
topsep=2pt plus 2pt, labelindent=.2em, leftmargin=*, widest=iii]}{\end{enumorig}}

\makeatletter 
\let\old@setaddresses\@setaddresses 
\def\@setaddresses{\bigskip\bgroup\parindent 0pt\let\scshape\relax\old@setaddresses\egroup}
\makeatother

\begin{document}

\title[Boolean dimension and dim-boundedness: Planar cover graph with a zero] 
{Boolean dimension and dim-boundedness:\\Planar cover graph with a zero}

\author[H.~S.~Blake]{Heather Smith Blake}
\address[H.~S.~Blake]{Mathematics \& Computer Science Department\\
  Davidson College\\ 
  Davidson, North Carolina 28035}
\email{hsblake@davidson.edu}

\author[P.~Micek]{Piotr Micek}
\address[P.~Micek]{Theoretical Computer Science Department\\ 
  Faculty of Mathematics and Computer Science\\ 
  Jagiellonian University\\ 
  Kraków, Poland}

\email{piotr.micek@uj.edu.pl}

\author[W.~T.~Trotter]{William T. Trotter}
\address[W.~T.~Trotter]{School of Mathematics\\ 
  Georgia Institute of Technology\\ 
  Atlanta, Georgia 30332}
\email{trotter@math.gatech.edu}

\thanks{P.\ Micek is partially supported by a Polish National Science 
 Center grant (BEETHOVEN; UMO-2018/31/G/ST1/03718).  H.\ S. Blake
 and W.\ T. Trotter are supported by grants from the Simons Foundation.}

\subjclass[2010]{06A07} 
\date{November 14, 2022; proposed revision October 12, 2026}

\keywords{Poset, planar graph, dimension, standard example}

\begin{abstract}
  In 1989, \NESPUD\ posed the following challenging question: Do planar posets 
  have bounded Boolean dimension?  We show that every poset with a planar cover 
  graph and a unique minimal element has Boolean dimension at most~$10$. 
  As a consequence, we are able to show that there is a reachability labeling 
  scheme with labels consisting of $\Oh(\log n)$ bits for planar digraphs 
  having a vertex from which every vertex is reachable. 
  The best known scheme for general planar digraphs 
  uses labels with $\Oh(\log^2 n)$ bits 
  [Thorup, J.\ ACM 2004], and it remains open to determine whether a scheme using 
  labels with $\Oh(\log n)$ bits exists. 
  The Boolean dimension result is proved in tandem with a second result showing that
  the dimension of a poset with a planar cover graph and a unique minimal 
  element is bounded by a linear function of its standard example number.
\end{abstract} 

\maketitle

\section{Introduction}\label{sec:introduction}

\subsection{Dimension and Boolean dimension}

All posets and digraphs in this paper are finite, and posets are nonempty. 
Partially ordered sets, called \emph{posets} for short, are
combinatorial structures with applications in many areas of mathematics and 
theoretical computer science. 
The most widely studied measure of a poset's complexity is
its dimension, as defined by Dushnik and Miller~\cite{DM41}.  
A \emph{linear extension} $L$ of a poset $P$ is a total order on 
the elements of $P$ such that if $x\leq y$ in $P$, then $x\leq y$ in $L$. 
A \emph{realizer} of a poset $P$ is a set $\set{L_1,\ldots,L_d}$ of
linear extensions of $P$ such that
\[
  x \leq y\Longleftrightarrow (x\leq y\text{ in $L_1$}) 
  \wedge \cdots \wedge (x\leq y\text{ in $L_d$}),
\]
for all $x, y\in P$.  The \emph{dimension}, denoted by $\dim(P)$, is 
the minimum size of a realizer of $P$.  

Realizers provide a compact scheme for handling comparability queries: 
Given a realizer $\set{L_1,\dots,L_d}$ for a poset $P$, then a
query of the form "is $x\leq y$?" can be answered by looking at the
relative position of $x$ and $y$ in each of the $d$ linear extensions
of the realizer. 

As noted by Gambosi, \NES\ and Talamo~\cite{GNT90}, these observations 
motivate a more general encoding of posets.  
A \emph{Boolean realizer} of a poset $P$ is a sequence of linear orders, 
i.e.\ not necessarily linear extensions, 
$(L_1,\ldots,L_d)$ of the elements of $P$
and a $d$-ary Boolean formula $\phi$ such that
\[
  x \leq y \Longleftrightarrow \phi((x\leq y\text{ in $L_1$}),\ldots,(x\leq y\text{ in $L_d$}))=1,
\]
for every $x, y\in P$. 
The \emph{Boolean dimension} of $P$, denoted by $\bdim(P)$, is a
minimum size of a Boolean realizer.  Clearly, for every
poset $P$ we have $\bdim(P) \leq \dim(P)$.

As is well known, when $n\ge4$, the maximum dimension of a poset on $n$ elements
is $\floor{n/2}$. The upper bound in this statement is evidenced
by the following construction:
When $d\ge2$, the \emph{standard example} $S_d$ is a poset
whose ground set is $\{a_1,\ldots,a_d, b_1,\ldots,b_d\}$ with $a_i < b_j$
in $S_d$ if and only if $i\neq j$ (see Figure~\ref{fig:stand6}). 
Dushnik and Miller~\cite{DM41} observed that $\dim(S_d) = d$.  
On the other hand, it is an instructive exercise to show that 
$\bdim(S_d)\leq4$ for every $d\ge2$.  In~\cite{NP89}, \NESPUD\ showed that 
the maximum Boolean dimension among posets on $n$ elements is $\Theta(\log{n})$. 

\begin{figure}
  \centering
  \includegraphics[scale=.8]{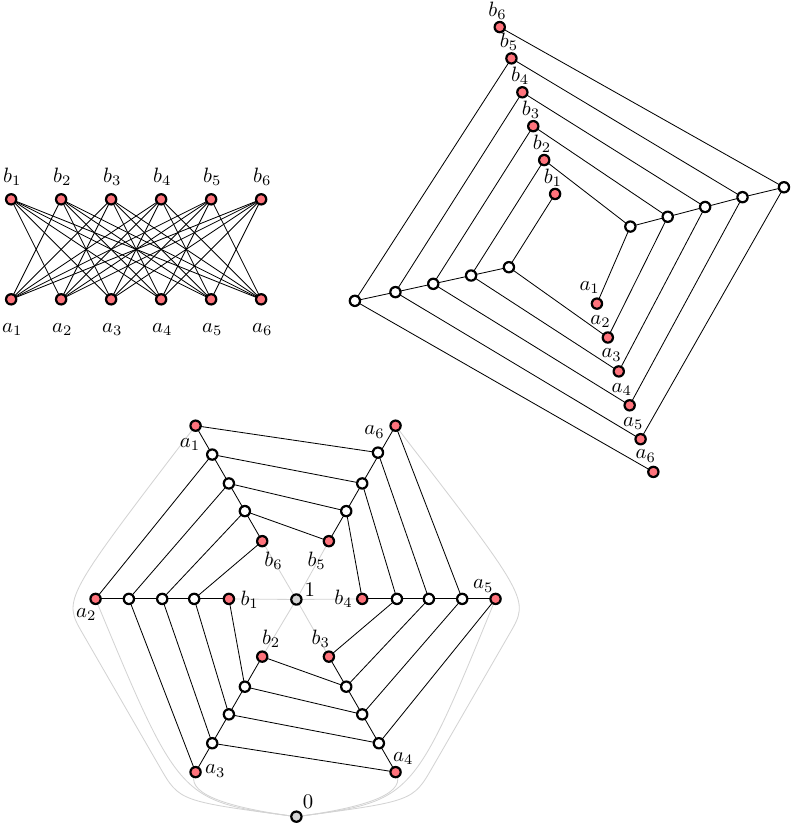}
  \caption{\label{fig:stand6} The standard example $S_6$ (left). Kelly's
  planar poset containing $S_6$ as a subposet (right). 
  These two examples are drawn in diagram form, i.e., 
  the comparabilities flow monotonically upwards along the edges. 
  The wheel construction of the poset with a planar cover graph, with a 
  single minimal element and a single maximal element, and containing $S_6$ 
  as a subposet (bottom). Note also that at the bottom we have a planar drawing 
  of the cover graph (not diagram) and the poset relation goes ``inwards''.}
  \label{fig:Kelly-wheel}
\end{figure}

The \emph{cover graph} of a poset $P$ is the graph whose vertices are the
elements of $P$ with edge set $\{xy\mid x<y \text{ in $P$ and there
is no $z$ with $x<z<y$ in $P$}\}$.  
Somewhat unexpectedly, posets with planar cover graphs can have
arbitrarily large dimension.  In 1978, Trotter~\cite{Tro78} gave the ``wheel
construction'' illustrated in Figure~\ref{fig:Kelly-wheel}. This construction
shows that there are posets with planar cover graphs, a unique minimal element,
a unique maximal element, and arbitrarily large dimension.  In 1981,
Kelly~\cite{Kel81} gave a construction (also illustrated in Figure~\ref{fig:Kelly-wheel})
that shows that there are posets with planar order diagrams that have 
arbitrarily large dimension.   In both of these constructions, the fact that
dimension is large stems from the fact that the posets contain large standard examples.
Note that the posets in Kelly's construction have a large number of 
minimal elements.  However, their cover graphs have pathwidth 
at most $3$.  On the other hand, the posets in Trotter's construction have a 
unique minimal element, but their cover graphs have unbounded treewidth. 
We also note that the Boolean dimension of the posets in both constructions
is bounded. 

Here is an intriguing question posed in~\cite{NP89} 
that remains unanswered to this day:
\begin{problem}[Ne\v{s}et\v{r}il, Pudl\'ak 1989]\label{problem:NESPUD}
  Do posets that have planar cover graphs\footnote{%
    The original question was posed for posets with planar diagrams.
    We ask a slightly more general question to establish an immediate connection with 
    reachability labelings.} have bounded Boolean dimension?
\end{problem}

Ne\v{s}et\v{r}il and Pudl\'{a}k suggested an approach for a
negative resolution of this question that involves an auxiliary 
Ramsey-type problem for planar posets.  However, no progress in this direction
has been made.
From the positive side,  researchers have in recent years investigated 
conditions on cover graphs that bound
Boolean dimension.  M\'{e}sz\'{a}ros, Micek and Trotter~\cite{MMT19} proved that 
the Boolean dimension of a poset is bounded in terms of the Boolean dimension of 
its $2$-connected blocks.  Felsner, M\'{e}sz\'{a}ros, and Micek~\cite{FMM20} 
proved that posets with cover graphs of bounded treewidth have bounded 
Boolean dimension.  
On the other hand, as noted previously, the Kelly examples have 
unbounded dimension, but their cover graphs have pathwidth at most $3$.

The first of the two principal theorems in this paper is:
\begin{theorem}
\label{thm:bdim-10}\label{thm:bdim-11}
If $P$ is a poset with a planar cover graph and a unique minimal element, then
\[
    \bdim(P) \leq 10.
\]
\end{theorem}

\subsection{Dim-boundedness.} 
As referenced in the abstract, our proof of Theorem~\ref{thm:bdim-10} is proved
in tandem with a second result that is relevant to a conjecture that is even older
than Problem~\ref{problem:NESPUD}. The \textit{standard example number} of a poset $P$, 
denoted $\se(P)$, is set to be~$1$ if $P$ does not contain a subposet isomorphic 
to the standard example~$S_2$; otherwise $\se(P)$ is the largest 
$d\ge2$ such that $P$ contains a subposet isomorphic to the standard example $S_d$.
Obviously, a poset that contains a large standard example has large dimension,
i.e., $\dim(P)\ge \se(P)$ for every poset $P$.  As is well known, if $P$ is
a distributive lattice, and $\dim(P)\ge3$, then $\dim(P)=\se(P)$, so there are 
important classes of posets for which the inequality $\dim(P)\ge\se(P)$ is tight.

However, dimension can be large without the presence of standard examples. 
The class of posets with standard example number~$1$ is the
class of \textit{interval orders}, and this special class of posets has been studied
extensively in the literature. As noted in~\cite{BHPT16},  
the maximum dimension of an interval order on $n$ elements is known to be 
within $\Oh(1)$ of $\log\log n+\frac{1}{2}\log\log\log n$. 

In 2020, Bir\'{o}, Hamburger, Kierstead,
P\'{o}r, Trotter and Wang~\cite{BHKPTW20} used random methods for
bipartite posets to show that 
there is a constant $n_0$ such that 
for every $n>n_0$, there is an $n$-element poset $P$ with $\se(P)=2$ and 
$\dim(P)> \frac{n^{1/6}}{8\log n}$.
Moreover, for every integer $d\geq3$ there is a constant $\alpha_d>0$ 
so that $\displaystyle\lim_{d\to\infty} \alpha_d= 1$, and
the maximum dimension of $n$-element posets $P$ with $\se(P)<d$ is 
$\Omega\left(n^{\alpha_d}\right)$.  
Upper bounds are more challenging.  In 2015, Bir\'{o}, Hamburger and 
P\'{o}r~\cite{BHP15} proved that for a fixed value of $d\ge3$,
the maximum dimension among posets on $n$ elements that have standard
example number less than $d$ is $o(n)$.
It is a great challenge to determine whether the $o(n)$ upper bound can be replaced
by $\Oh(n^{\alpha})$ with some $\alpha<1$ depending only on $d$.

A class $\cgC$ of posets is 
\emph{$\dim$-bounded} if there is a function $f:\bbN\rightarrow\bbN$ such
that $\dim(P)\le f(\se(P))$ for every poset $P$ in $\cgC$. 
There is an analogous research theme in graph theory:  When $G$ is
a graph, $\chi(G)$ denotes the \textit{chromatic number} of $G$, and
$\omega(G)$ denotes the \textit{clique number} of $G$.  In this setting,
we have the trivial inequality: $\chi(G) \ge \omega(G)$.  
As is well known, there are triangle-free graphs with arbitrarily large
chromatic number.  Regardless, researchers have found interesting classes of 
graphs where chromatic number is bounded in terms of clique number.  Such 
classes are called $\chi$-\textit{bounded}. A recent survey
by Scott and Seymour~\cite{SS20} has just appeared, and this paper gives 
an excellent summary of the substantial progress on $\chi$-boundedness 
achieved within the last decade.

The following conjectures motivated this work.
\begin{conjecture}\hfill
  \label{con:dim-bounded}
  \begin{enumerate}
  \item The class of posets that have planar diagrams is $\dim$-bounded.
  \label{item:conj:diagram}
  \item The class of posets that have planar cover graphs is $\dim$-bounded.
  \label{item:conj:cover-graph}
\end{enumerate}
\end{conjecture}

We believe the first published reference to Conjecture~\ref{con:dim-bounded}.\ref{item:conj:diagram} is
an informal comment on page 119 in~\cite{Tro-book} published
in 1992.  However, the conjecture was circulating among researchers soon
after the constructions illustrated in Figure~\ref{fig:Kelly-wheel}
appeared. 
Conjecture~\ref{con:dim-bounded}.\ref{item:conj:cover-graph} is the stronger statement.
Since the first version of this paper, Blake, Hodor, Micek, Seweryn and
Trotter~\cite{BHMST25} have proved that posets with planar cover graphs satisfy
$\dim(P)=\Oh(\se(P)^8)$, resolving both conjectures.  Their subsequent
work~\cite{BHMST26} gives $\dim(P)\le96\se(P)+672$ for posets with planar
diagrams.  

The theorem below gives a sharper linear bound under the
additional hypothesis of a unique minimal element. This is our second principal contribution.
\begin{theorem}\label{thm:dim-boundedness}
  The class of posets
  with a planar cover graph and a unique minimal element is $\dim$-bounded. 
  Specifically, given a poset $P$ in this class, we have
  \[
    \dim(P)\le 2\se(P)+2.
  \]
\end{theorem}
Since $\dim(P)\geq\se(P)$ for all $P$, 
the upper bound in Theorem~\ref{thm:dim-boundedness} is best possible 
to within a multiplicative factor of $2$.

The remainder of the paper is organized as follows.  In 
the next section, we give a brief treatment of reachability labeling schemes and the
connection with Boolean dimension.  This is followed by a
section with essential preparatory material on posets, dimension and Boolean
dimension.  The arguments for each of our two main theorems can be split
into two parts.  The first part is the same for both theorems, and this
common part will be contents of Section~\ref{sec:common}.  
This will be followed
by sections giving the second parts of the proofs for each of the two main theorems,
with the result for Boolean dimension given first.  

\section{Reachability labeling schemes and Boolean dimension}\label{sec:reachability}

Given two vertices $u$, $v$ in a directed graph (digraph, for short), 
we say that $u$ can \emph{reach} $v$ 
if there is a directed path, possibly of length zero, from $u$ to $v$ in the digraph.
A class of digraphs $\cgC$ admits an 
\emph{$f(n)$-bit reachability scheme} if there exists a function 
$A:(\{0,1\}^*)^2\to\{0,1\}$
such that for every positive integer $n$ and every $n$-vertex digraph $G\in \cgC$ 
there exists $\ell:V(G)\to\{0,1\}^*$ such that 
$|\ell(v)|\leq f(n)$ for each vertex $v$ of $G$, and such that, for every two 
vertices $u$, $v$ of $G$

\[  A(\ell(u),\ell(v)) =
\begin{cases}
  1 & \text{if $u$ can reach $v$ in $G$;} \\
  0 & \text{otherwise.}
\end{cases}
\]
Bonamy, Esperet, Groenland, and Scott~\cite{BEGS21}
have recently devised a reachability labeling scheme
for the class of \emph{all} digraphs using labels of length at most $n/4+o(n)$.
This is best possible up to the lower order term as 
a simple counting argument forces 
every reachability labeling scheme 
for the class of \emph{all} $n$-vertex digraphs 
to use a label of length at least $n/4-o(n)$.

Quoting from~\cite{HRT15} ``\emph{In this paper we focus on the planar 
case, which feels particularly relevant when you live on a sphere.''}
In 2004, Thorup~\cite{Tho04} presented an $\Oh(\log^2 n)$-bit reachability 
labeling scheme for planar digraphs.  It remains open to answer whether 
more efficient schemes exist.
\begin{problem}\label{prob:labeling}
  Do planar digraphs admit an $\Oh(\log n)$-bit reachability labeling scheme?
\end{problem}

Planar acyclic digraphs with a single source $s$ and a single sink~$t$ admitting a plane
drawing with both $s$ and $t$ on the exterior face are called $st$-\emph{planar graphs.}
These graphs have a very simple structure, and within this class, it is straightforward 
to construct a $\Oh(\log n)$-bit reachability labeling scheme. 
This observation was an important tool in related work on reachability oracles by 
Holm, Rotenberg, and Thorup~\cite{HRT15}.

There is a standard technique for reducing reachability queries in digraphs to comparability 
queries in posets:\quad  Given a digraph $G$, contract each strongly connected 
component of $G$ to a single vertex. 
Let $G'$ be the resulting digraph which is 
obviously acyclic. Note also that if $G$ is planar, 
then $G'$ is planar as well.
Now given a labeling of $G'$ we extend it to a labeling of $G$
by assigning to each vertex $v$ of $G$ 
the label 
of the strong component of $v$ in the labeling of $G'$. 
Within an acyclic digraph $G'$, for all $u$, $v$ in $G'$, set $u\leq v$ if $u$ can reach $v$ in $G'$. 
Clearly $(G',\leq)$ forms a partially ordered set. 
Every cover relation is an arc of $G'$, so its cover graph is a subgraph of the underlying graph of $G'$ and is planar. 
Thus, if we have an $f(n)$-bit comparability labeling scheme for posets with planar cover graphs, 
then we immediately get an $f(n)$-bit reachability scheme for general planar digraphs.

Note that if $\mathcal{C}$ is a class of posets with bounded Boolean dimension,
then $\mathcal{C}$ admits an $\Oh(\log n)$ comparability labeling scheme.
To see this, suppose that $\bdim(P)\le d$ for every poset $P$ in $\mathcal{C}$.
Now let $P$ be in $\mathcal{C}$, and let $(L_1,\ldots,L_d)$ with a formula $\phi$ be a Boolean realizer of $P$.
Let $n$ be the number of elements of $P$.  We label each element $x\in P$ with a 
bitstring of length $d\cdot\lceil\log n\rceil$ 
describing the positions of $x$, numbered $0,\ldots,n-1$, in $(L_1,\ldots,L_d)$. 
Now given labels for two elements $x, y\in P$ and the 
formula $\phi$, 
we can determine if $x \leq y$ in $P$. 
The formula $\phi$ is a function from $\{0,1\}^d$ to $\{0,1\}$, so there 
are only $2^{2^{d}}$ possibilities, and these can be encoded with an additional 
$2^{d}$ bits in each label.
We note that the proof of Theorem~\ref{thm:bdim-10} gives 
us conveniently the same $\phi$ for every poset $P$ 
with a planar cover graph and 
a unique minimal element, so these extra bits are not necessary.  

Therefore, we have the following corollary to Theorem~\ref{thm:bdim-10}, which 
we believe represents an important step towards a positive resolution of 
Problem~\ref{prob:labeling}.
\begin{corollary}\label{cor:reachability-scheme}
  Planar digraphs having a vertex from which every vertex is reachable 
  admit an $\Oh(\log n)$-bit reachability labeling scheme.
\end{corollary}

\section{Notation, Terminology and Essential Background Material}\label{sec:preliminaries}

We write $[k]$ as a
compact form of $\{1,\dots,k\}$.  

When $u$ and $v$ are distinct vertices in a tree $T$,
we denote by $uTv$ the unique path in $T$ from $u$ to $v$. 
This notation allows for the natural notion of concatenation when
working with paths, i.e., 
when $N$ and $N'$ are paths in a graph $G$, $a,b\in N$ and $b,c\in N'$,
then $aNbN'c$ denotes a walk in $G$ formed by concatenating $aNb$ with $bN'c$. 
When this convention is applied later in this paper, the
paths $aNb$ and $bN'c$  will have no vertices in common other than
$b$.  As a result, $aNbN'c$ will also be a path.


Let $P$ be a poset.  
The \emph{dual of $P$} is the poset with the same ground set as $P$ and with $x \leq y$ in the dual of $P$ if and only if $y\leq x$ in $P$.
When $a$ and $b$ are incomparable in $P$, we will write $a\parallel b$ in $P$,
and sometimes we will just use the short form
$a\parallel _P b$.  Analogous short forms will be
used for $<,>,\le,\ge$.  When $x\in P$, we let $U_P(x)$ consist of all
$v\in P$ such that $x<_P v$.  Also, we set $U_P[x]=U_P(x)\cup\{x\}$.
Of course, the statement that $x_0$ is the unique minimal element of
$P$ means the same as saying $U_P[x_0]$ 
contains all elements of $P$.  Analogously, $D_P(x)$ consists
of all $u\in P$ such that $u<_P x$, and $D_P[x]=D_P(x)\cup\{x\}$.

When $xy$ is an edge in the cover graph of $P$ and $x<_P y$, 
we say that $x$ is \emph{covered} by $y$ in $P$, or 
$y$ \emph{covers} $x$ in $P$.  
This constitutes a natural acyclic orientation of the cover graph of $P$,  
which will be used implicitly.

When $P$ is a poset, $x\in P$,
and $S$ is a subset of $P$ that does not contain $x$, we will
write $x\parallel_P S$ when $x\parallel_P y$ for every $y\in S$.
The notations $x<_P S$ and $x>_P S$ are defined analogously.

A \emph{witnessing path $W$ in $P$} is a sequence $W=(u_0,\ldots,u_m)$ of 
elements in $P$ such that if $m>0$, then $u_i$ is \emph{covered} by $u_{i+1}$ in $P$ whenever $0\le i< m$.  
We will say in this case that $W$ is a witnessing path 
\emph{from} $u_0$ \emph{to} $u_m$. 
In discussions on sequences of elements,  we will use the terms prefix and suffix, 
i.e.,
when $(z_0,\dots,z_m)$ is a sequence, and $0\le i\le m$, 
the subsequence $(z_0,\dots,z_i)$
is a \emph{prefix} and the subsequence $(z_i,\dots,z_m)$ is a \emph{suffix} 
of the initial sequence.

Members of the ground set of a poset will be called
elements and points interchangeably. Also, when we are discussing a graph on
the same ground set as a poset, points and elements will also be called vertices.

Let $P$ be a poset and let $\Inc(P)$ denote the set of all pairs $(a,b)$ 
such that $a\parallel_P b$. 
Let $L$ be a linear extension of $P$ and let $(a,b)\in\Inc(P)$. 
We say that $L$ \emph{reverses} $(a,b)$ 
if $b<a$ in $L$. 
When $S$ is a subset of $\Inc(P)$, 
we say that $L$ \emph{reverses $S$}  
if $L$ reverses every pair in $S$. 
A subset $S$ of $\Inc(P)$ is \emph{reversible in $P$}
if there is a linear extension $L$ of $P$ such that
$L$ \emph{reverses $S$}. 
Now the dimension of $P$ 
can be redefined as follows: 
if $\Inc(P)$ is empty, then $\dim(P)=1$; otherwise,  
$\dim(P)$ is the least positive integer $t$ for which 
$\Inc(P)$ can be covered by $t$ sets which are reversible in $P$.    

A sequence $((x_1,y_1), \dots, (x_k,y_k))$ of pairs from $\Inc(P)$ 
with $k \geq 2$ is an \emph{alternating cycle of size $k$ in $P$} 
if $x_i\leq_P y_{i+1}$ for all $i\in\set{1,\ldots,k}$, cyclically 
(so $x_k\le_P y_1$ is required).  
Observe that if $((x_1,y_1), \dots, (x_k,y_k))$ is an alternating 
cycle in $P$, then any subset $I\subseteq \Inc(P)$ containing
all the pairs on this cycle is not reversible in $P$.

An alternating cycle $((x_1,y_1),\dots,(x_k,y_k))$ in $P$ is \emph{strict} if we have
$x_i\le_P y_{j}$ if and only if $j=i+1$ (cyclically).
Note that in this case, $\{x_1, \dots, x_k\}$ and
$\{y_1, \dots, y_k\}$ are $k$-element antichains.  Note also that
in alternating cycles, we allow that $x_i=y_{i+1}$ for some or even
all values of $i$.  In~\cite{TM77}, Trotter and Moore made the following
elementary observation that has proven over time to be far reaching in nature:
A subset $I\subseteq\Inc(P)$ is reversible in $P$ if and only 
if $I$ contains no strict alternating cycle in $P$.


Our primary focus will be on posets that have planar cover graphs.  When
$P$ has a planar cover graph $G$, 
we fix a drawing without edge crossings of
$G$ in the plane, we follow the standard convention that also considers an
element $x$ of the poset $P$ as a point in the plane.  Similarly, an edge or a 
path in $G$ will be considered as a subgraph of $G$ and a simply connected set of 
points in the plane.  When $N$ is a path in $G$, the vertices and edges of $N$ form
a simply connected set of points in the plane.  When $D$ is a cycle in $G$, the points 
in the plane belonging to the edges of $D$ form a simple closed curve.  In this 
case, we abuse notation by simply saying that a point in the plane is either on $D$, 
in the interior of $D$ or in the exterior of $D$, without further reference to the curve.


We start with a compact summary of material which is largely taken from the
paper~\cite{KMT21+} by Kozik, Micek and Trotter.
Let $P$ be a poset with a planar cover graph $G$ and 
with a unique minimal element $x_0$.

Since $x_0$ is an element of $P$, we can fix a drawing without edge crossings
of $G$ with $x_0$ on the exterior face.  To assist in ordering edges in arguments
to follow, we append an imaginary edge $e_{-\infty}$ attached to $x_0$ in the 
exterior face.  This setup is illustrated in Figure~\ref{fig:setup},
and we will return to this example several times later in the paper.

\begin{figure}[!h]
  \begin{center}
    \includegraphics[scale=.8]{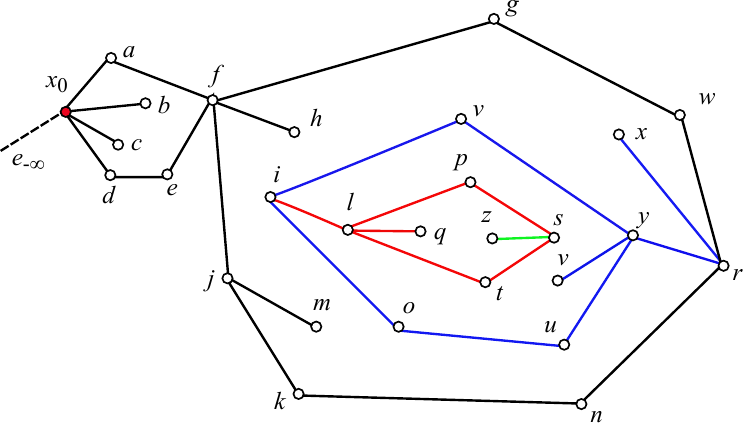}
  \end{center}
  \caption{
    A poset with a planar cover  and a unique minimal element $x_0$ drawn in the exterior face. An imaginary edge $e_{-\infty}$ is attached to $x_0$.
    In this figure, 
    the black and red edges are oriented left-to-right in the plane, 
    while the blue and green edges are oriented right-to-left.
  }
  \label{fig:setup}
\end{figure}


When $z$ is an element of $P$, 
there is a natural clockwise cyclic ordering of 
the edges of $G$ incident with $z$. 
When $e_0$, $e$, and $e'$ are edges (not necessarily distinct) 
incident to $z$, we will write $e_0\preceq e\preceq e'$, if 
starting with edge $e_0$ and proceeding 
in a clockwise manner around $z$, stopping at $e'$, 
we have visited the edge $e$. 
We can replace one or both of the $\preceq$ symbols with 
$\prec$ when the corresponding edges are distinct.

When $e_0$ is a particular edge incident to $z$, 
there is a clockwise \emph{linear} order on the edges incident with $z$ 
with $e_0$ the least element.  If $e$ and $e'$ are distinct edges and
neither is $e_0$, we say that $e\prec e'$ in the \emph{$(z,e_0)$-ordering} if 
$e_0\prec e\prec e'$.  To reinforce the geometric implications, 
we will also say that $e$ is \emph{left of} $e'$ in the $(z,e_0)$-ordering. 
Of course, we will also say that $e'$ is \emph{right of} $e$ in the
$(z,e_0)$-ordering when $e$ is left of $e'$ in the $(z,e_0)$-ordering.
The $(z,e_0)$-ordering on edges is illustrated in Figure~\ref{fig:ze-order}.


Let $u$ and $u'$ be
(not necessarily distinct) elements of $P$. 
Also, let $W$ and $W'$ be paths (not necessarily witnessing paths in $P$) in $G$ 
from $x_0$ to $u$ and $u'$, respectively.  We say that $W$ and $W'$ are 
$x_0$-\emph{consistent} if there is an element $z$ common to $W$ and $W'$ 
such that 
(1)~$x_0Wz=x_0W'z$; and 
(2)~$zWu$ and $zW'u'$ are disjoint except their common starting element $z$.
Note that $W$ and $W'$ are $x_0$-consistent whenever one is a prefix of the
other.  In fact, a path $W$ from $x_0$ to an element $u$ in $P$ 
is $x_0$-consistent with itself.

Now suppose that $W$ and $W'$ are $x_0$-consistent and neither is a prefix of
the other.  Let $z$ be the last vertex of the common prefix of $W$ and $W'$, and let $e_0$ be the edge immediately before $z$ on the
path $x_0Wz$ (let $e_0=e_{-\infty}$ if $z=x_0$).
Let $e$ and $e'$ be the first edges of $zWu$ and $zW'u'$, respectively. 
We say that $W$ is $x_0$-\emph{left} of $W'$ if 
$e$ is left  of $e'$ in the $(z,e_0)$-ordering.  Symmetrically,
we say that $W$ is $x_0$-\emph{right} of $W'$ if
$e$ is right of $e'$ in the $(z,e_0)$-ordering.
Note that if $W$ and $W'$ are $x_0$-consistent, and neither is a prefix of the
other, then either $W$ is $x_0$-left of $W'$, or $W$ is $x_0$-right of $W'$.

\begin{figure}[!h]
  \begin{center}
    \includegraphics[scale=.8]{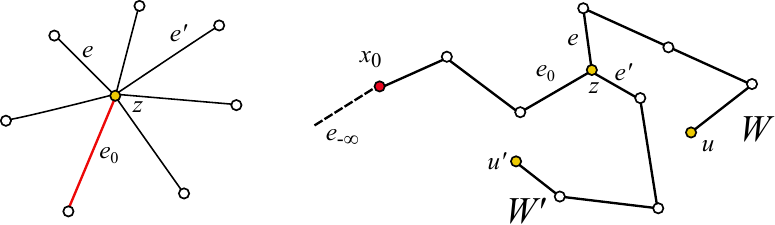}
  \end{center}
  \caption
  {Left:  For the vertex $z$ and the fixed edge $e_0$ incident with $z$,
  edge $e$ is left of $e'$ in the $(z,e_0)$-ordering. 
  Right: The paths $W$ and $W'$ are $x_0$-consistent and $W$ is $x_0$-left of $W'$.
  }
  \label{fig:ze-order}
\end{figure}
\section{The Common Part of the Proofs of our Two Main Theorems}\label{sec:common}

Let $P$ be a poset with a planar cover graph and a unique minimal element.
We will quickly develop a classification scheme that results in every incomparable
pair in $P$ being labeled as one of four types: \emph{left}, \emph{right}, 
\emph{outside}, or \emph{inside}.  We will then show that there
are linear extensions $L_1$ and $L_2$ of $P$ such that if $(a,b)\in \Inc(P)$,
and $(a,b)$ is a left pair, a right pair or an outside pair, then
$b<a$ in either $L_1$ or $L_2$.  

Now suppose our focus was on Boolean dimension.  Given a pair $(a,b)$ of distinct
points in $P$, our goal is to answer with full certainty the question:
``Is $a$ less than or equal to $b$ in $P$?''.  Suppose that we know the answer as to whether
$a\leq b$ in $L_i$ for each $i\in[2]$.  If any one of these answers is negative, then
we \emph{know} that $a$ is \emph{not} less than or equal to $b$ in $P$.  Moreover, if the answer is positive for each $i\in[2]$, then we know 
that one of the following two statements hold: 
(1)~$a$ is less than $b$ in $P$;
or (2)~$(a,b)$ is an incomparable pair that is an inside pair.   
Our goal for part~2 of the
proof for Theorem~\ref{thm:bdim-10}, given in Section~\ref{sec:bdim}, will be to show that 
we can construct $8$ additional linear orders $(L_3,\dots,L_{10})$,
so that if we also know the answer to whether $a$ is less than or equal to $b$ in
$L_j$ for $j\in\{3,\dots,10\}$, then we can answer with complete confidence the
question as to whether $a$ is less than or equal to $b$ in $P$.  

Now suppose that our focus was on the issue of $\dim$-boundedness for posets with
planar cover graphs and a unique minimal element.  In the second part of this
proof, given in Section~\ref{sec:dim-bounded},
we will then show that the set of all inside pairs can be covered with
$2\se(P)$ sets which are reversible in $P$.  

Now we begin our work in the common part of the proof by
developing a series of propositions and lemmas 
concerning posets that have a unique minimal element and a planar cover graph.
Accordingly, we continue with the same setup used in the preceding section
(and as illustrated in Figure~\ref{fig:setup}), i.e.,
we assume:
\begin{enumerate}

  \item $P$ is a poset and $x_0$ is the unique minimal element of $P$.
  \item $G$ is the cover graph of $P$ and $G$ is planar.
  \item We fix a drawing without edge crossings of $G$ in the plane with 
    $x_0$ on the exterior face.
  \item To provide a base edge among those incident in $G$ with $x_0$, we add 
    to the drawing an ``imaginary'' edge $e_{-\infty}$ in the exterior face, 
    with $e_{-\infty}$ incident with $x_0$.
\end{enumerate}
Let $u$ be an element of $P$.
The \emph{leftmost witnessing path from $x_0$ to $u$},
denoted by $\Wleft(u)$, is 
constructed using the following inductive procedure: 
If $u=x_0$ then $\Wleft(u)$ is a one-vertex path containing $u$. 
Otherwise, start with $u_0=x_0$ and $e_0=e_{-\infty}$.  Then $u_0<_P u$. 
Now suppose that for some $i\ge0$, we have defined a witnessing
path $W_i=(u_0,\dots,u_i)$, $u_i< u$ in $P$, and $e_i$ is the edge
immediately before $u_i$ on $W_i$.  Among the edges incident
with $u_i$, some may be directed towards $u_i$, but since
$u_i<_P u$, there is a non-empty set $E_i$ of edges incident
with $u_i$ that are directed away from $u_i$ and 
lie on a witnessing path from $u_i$ to $u$.  We let $e_{i+1}$ be the 
leftmost edge in $E_i$ in the $(u_i,e_i)$-ordering, and we
take $u_{i+1}$ as the other end point of $e_{i+1}$.
When $u_{i+1}=u$, the procedure halts and outputs the witnessing path
$(u_0,\dots,u_{i+1})$ as $\Wleft(u)$.  For the poset shown
in Figure~\ref{fig:setup}, we illustrate 
the leftmost path from $x_0$ to $z$ in Figure~\ref{fig:left-most-path}.

\begin{figure}[!h]
  \begin{center}
    \includegraphics[scale=.8]{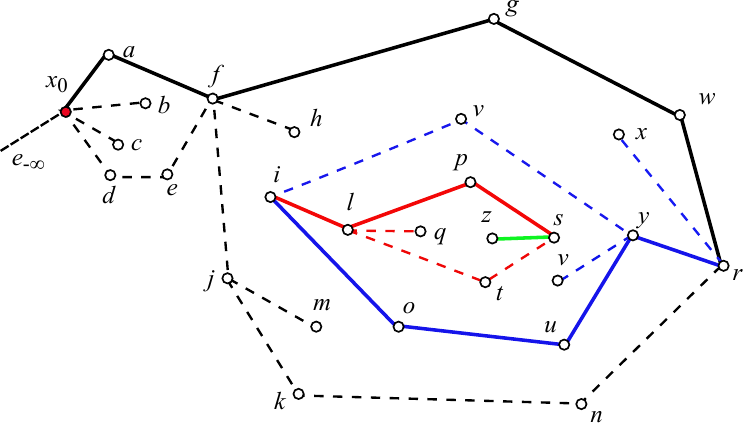}
  \end{center}
  \caption
  {We show $\Wleft(z)$, the leftmost witnessing path from $x_0$ to $z$, using
  bold edges.  The remaining edges in the cover graph are dashed. 
  In this figure, 
    the black and red edges are oriented left-to-right in the plane, 
    while the blue and green edges are oriented right-to-left.
  }
  \label{fig:left-most-path}
\end{figure}
When $u$ is an element of $P$, the \emph{rightmost witnessing path from $x_0$ to $u$}, 
denoted by $W_R(u)$, is defined in a symmetric manner.
\begin{proposition}\label{pro:Tleft-and-Tright}
  Let $u$ and $u'$ be elements of $P$. 
  Then $\Wleft(u)$ and $\Wleft(u')$ are $x_0$-consistent. 
  Also, $\Wright(u)$ and $\Wright(u')$ are $x_0$-consistent.
\end{proposition}
\begin{proof} 
  We use an argument by contradiction to show that $\Wleft(u)$ and $\Wleft(u')$ 
  are $x_0$-consistent. The argument for $\Wright(u)$ and $\Wright(u')$ is symmetric.
  Let $z$ be the largest element of $P$ such that $x_0\Wleft(u)z=x_0\Wleft(u')z$.  Since
  $\Wleft(u)$ and $\Wleft(u')$ are not $x_0$-consistent,  (1)~$z\not\in\{u,u'\}$;
  (2)~there are distinct elements $w$ and $w'$ both of which cover $z$ in $P$ such
  that $\Wleft(u)$ contains the edge $e=zw$ and $\Wleft(u')$ contains the edge $e'=zw'$;
  and (3)~there is an element $z'\in P$ common to $\Wleft(u)$ and $\Wleft(u')$ such
  that $\{w,w'\}<_P z'$.  Note that $z'\le_P\{u,u'\}$.

  Let $e_0$ be the last edge common to $x_0\Wleft(u)z=x_0\Wleft(u')z$ 
  (again $e_0=e_{-\infty}$ if $z=x_0$).  Without loss of generality,  we may suppose that $e$ is 
  left of $e'$ in the $(z,e_0)$-ordering.  Now consider the choice the inductive
  procedure made in constructing
  $\Wleft(u')$ when it was at vertex $z$.  It considered all edges of
  the form $zv$ with $v$ covering $z$ in $P$ and $v\le_P u'$.  Among these
  edges, the algorithm chooses the edge which is least in the $(z,e_0)$-ordering.
  However, the edge $e=zw$ is among the choices available since $w< z'\le u'$ in $P$.  
  It follows that the procedure will not choose the edge $e'$.  The contradiction completes the
  proof.
\end{proof}
It follows from the above proposition that for any element $u$ of $P$, 
if $w \in W_L(u)$, then $W_L(w)$ is a prefix of $W_L(u)$. 
Similarly, if $w\in W_R(u)$, then $W_R(w)$ is a prefix of $W_R(u)$.

\begin{proposition}\label{pro:shortcuts} 
  Let $u$ and $v$ be elements of $P$.  Then the following statements hold:  
  \begin{enumerate} 
    \item\label{pro:item:shortcuts-left-tree} 
     If $\Wleft(u)$ is $x_0$-left of $\Wleft(v)$, then $u\not< v$ in $P$.  
    \item\label{pro:item:shortcuts-right-tree} 
      If $\Wright(u)$ is $x_0$-right of $\Wright(v)$, then $u\not< v$ in $P$.  
  \end{enumerate}
\end{proposition}
\begin{proof}
  We prove statement~\ref{pro:item:shortcuts-left-tree}.  
  The argument for statement~\ref{pro:item:shortcuts-right-tree} is symmetric.
  Let $z$ be the largest element of $P$ common to
  $\Wleft(u)$ and $\Wleft(v)$.  Also, let $e_0$ be the common edge of these two paths that
  is immediately before $z$ (let $e_0=e_{-\infty}$ if $z=x_0$).  Let $e=zw$ and $e'=zw'$ be the edges immediately after
  $z$ in $\Wleft(u)$ and $\Wleft(v)$, respectively.  Since $\Wleft(u)$ is $x_0$-left of 
  $\Wleft(v)$, we know that $e$ is left of $e'$ in the 
  $(z,e_0)$-ordering.

  Now suppose that the proposition fails with $u< v$ in $P$.  This implies 
  $w\leq u < v$ in $P$.  Therefore, in leaving $z$, the procedure for
  constructing $\Wleft(v)$ had available the edge $e$, but incorrectly chose the edge $e'$.
  The contradiction completes the proof.
\end{proof}

There are two versions of the following proposition, with the roles of left
and right interchanged.  We state and prove one of the two versions.
\begin{proposition}
\label{pro:piotrek} 
  Let $v$ and $z$ be elements of $P$ such that
  $z$ is on $\Wleft(v)$ with $z\neq v$.  Let $e^+$ and $e^-$ be the
  edges of $\Wleft(v)$ that are, respectively, immediately after and immediately
  before $z$ (if $z=x_0$ then $e^-=e_{-\infty}$).  If $z$ is covered by $u$ in $P$, and
  the edge $e=zu$ is left of $e^+$ in the $(z,e^-)$-ordering,  then 
  $\Wleft(u')$ is $x_0$-left of $\Wleft(v)$ for all $u'\in U_P[u]$.
\end{proposition}
\begin{proof}
  Consider the common prefix of 
  $\Wleft(u')$ and $\Wleft(v)$. We split into two cases: 
  either the last element $w$ of their common prefix satisfies $w<z$ in $P$, or $w\geq z$ in $P$.

  Assume first that $w<z$ in $P$. 
  Let $e_w^+$ and $e_w^-$ be the edges of $\Wleft(v)$ that are, respectively, 
  immediately after and immediately before $w$ 
  (if $w=x_0$ then $e_w^-=e_{-\infty}$). 
  Consider the construction process of $\Wleft(u')$ at element $w$. 
  It looks for the least edge $e'=ww'$ in the $(w,e^-_w)$-ordering such that 
  $w < w' \leq u'$ in $P$. 
  Since $e^+_w$ is a valid choice and the two paths $\Wleft(u')$ and $\Wleft(v)$ split at $w$, the path $\Wleft(u')$ must have found a better candidate. 
  Therefore, $\Wleft(u')$ is $x_0$-left of $\Wleft(v)$, in this case.

  Now assume that the common prefix of $\Wleft(u')$ and $\Wleft(v)$ contains $z$, 
  i.e., $w\geq z$ in $P$. Again, consider the construction process of $\Wleft(u')$, 
  this time at element $z$. 
  Since $e$ is left of $e^+$ in the $(z,e^-)$-ordering and 
  $z < u \leq u'$ in $P$, 
  we conclude that there are better choices than $e^+$ for $\Wleft(u')$ to continue from $z$. 
  Therefore, the two paths $\Wleft(u')$ and $\Wleft(v)$ split at $z$ and $\Wleft(u')$ is $x_0$-left of $\Wleft(v)$, as desired.
\end{proof}
It is easy to construct examples of witnessing paths $W,W',W''$ from $x_0$ to
$u,u',u''$ such that (1)~$W$ and $W'$ are $x_0$-consistent and $W$ is $x_0$-left of $W'$;
(2)~$W'$ and $W''$ are $x_0$-consistent and $W'$ is $x_0$-left of $W''$; and
(3)~$W$ and $W''$ are not $x_0$-consistent.  However, if we know that
$W$ and $W''$ are $x_0$-consistent, we have the following more palatable result. 
\begin{proposition}\label{pro:consistent-paths-transitive}
  Let $W,W',W''$ be witnessing paths from $x_0$ to $u,u',u''$, respectively, and
  suppose that all three paths are pairwise $x_0$-consistent.  Then
  the following statements hold:
  \begin{enumerate}
    \item
    \label{pro:consistent-transitive-left} 
      If $W$ is $x_0$-left of $W'$, and $W'$ is $x_0$-left of $W''$, then
      $W$ is $x_0$-left of $W''$.
    \item 
    \label{pro:consistent-transitive-right} 
    If $W$ is $x_0$-right of $W'$, and $W'$ is $x_0$-right of $W''$, then
      $W$ is $x_0$-right of $W''$.
  \end{enumerate}
\end{proposition}
\begin{proof}
  We give the argument for statement~\ref{pro:consistent-transitive-left}.
  The argument for statement~\ref{pro:consistent-transitive-right} is symmetric.
  Let $z$ be the largest element common to all three paths $W,W',W''$. 
  Let $e_0$ be the edge immediately before $z$ on all three paths 
  (in case $z=x_0$, let $e_0=e_{-\infty}$).   
  Then let $e,e',e''$ be the edge after $e_0$ on
  $W,W',W''$, respectively.  
  The cyclic ordering on these edges is:
  \[
    e_0\prec e\preceq e'\preceq e''\prec e_0.
  \]
  Furthermore, either $e\neq e'$ or $e'\neq e''$.  It follows that
  \[
    e_0\prec e\prec  e''\prec e_0.
  \]
  This together with the fact that $W$ and $W''$ are $x_0$-consistent 
  implies that $W$ is $x_0$-left of $W''$, as desired.
  With this observation, the proof is complete.
\end{proof}
\subsection{Four Types of Incomparable Pairs}

Let $u$ and $v$ be elements of $P$.  
\begin{itemize} 
  \item  We say $(u,v)$ is a \emph{left pair} in $P$ if\\ 
    \indent  $\Wleft(u)$ is $x_0$-left of $\Wleft(v)$ and 
    $\Wright(u)$ is $x_0$-left of $\Wright(v)$.  
  \item We say $(u,v)$ is a \emph{right pair} in $P$ if\\ 
    \indent $\Wleft(u)$ is $x_0$-right of $\Wleft(v)$ and 
    $\Wright(u)$ is $x_0$-right of $\Wright(v)$;
\end{itemize}
Note that $(u,v)$ is a left pair in $P$ if and only if 
$(v,u)$ is a right pair in $P$.
\begin{proposition}\label{pro:details-on-LR}
  Let $u$, $v$, $w$ be elements of $P$. Then the following statements
  hold:
  \begin{enumerate}
  \item
  \label{pro:item:left-shortcuts-inverse}
    If $u< v$ in $P$, then $(u,v)$ is not a left pair.
  \item
  \label{pro:item:right-shortcuts-inverse}
    If $u< v$ in $P$, then $(u,v)$ is not a right pair.
  \item
  \label{pro:item:comparable-not-left-right}
    If $u$ and $v$ are comparable, then $(u,v)$ is not a left pair,
    and $(u,v)$ is not a right pair.  
  \item
  \label{pro:item:left-pair-is-transitive} 
    If $(u,v)$ is a left pair, and $(v,w)$ is a left pair, then 
    $(u,w)$ is a left pair.  
  \item
  \label{pro:item:right-pair-is-transitive} 
    If $(u,v)$ is a right pair, and $(v,w)$ is a right pair, then 
    $(u,w)$ is a right pair.
  \end{enumerate}
\end{proposition}
\begin{proof}
  Statements~\ref{pro:item:left-shortcuts-inverse} and~\ref{pro:item:right-shortcuts-inverse} follow immediately from the two statements of
  Proposition~\ref{pro:shortcuts}.
  Statement~\ref{pro:item:comparable-not-left-right} is then an immediate consequence of statements~\ref{pro:item:left-shortcuts-inverse} and~\ref{pro:item:right-shortcuts-inverse}.

  By Proposition~\ref{pro:Tleft-and-Tright} we have that any two paths from $\{\Wleft(u),\Wleft(v),\Wleft(w)\}$ are 
  $x_0$-consistent and any two paths from $\{\Wright(u),\Wright(v),\Wright(w)\}$ 
  are $x_0$-consistent.  Statements~\ref{pro:item:left-pair-is-transitive} and~\ref{pro:item:right-pair-is-transitive} then follow by applying
  Proposition~\ref{pro:consistent-paths-transitive}. 
\end{proof}

Among the incomparable pairs of $P$, we characterize those that
are neither left pairs nor right pairs as follows.
Let $(u,v)\in\Inc(P)$.
\begin{itemize} 
  \item  We say $(u,v)$ is an \emph{outside pair} in $P$ if\\ 
    \indent  $\Wleft(u)$ is $x_0$-left of $\Wleft(v)$ and 
    $\Wright(u)$ is $x_0$-right of $\Wright(v)$.  
  \item We say $(u,v)$ is an \emph{inside pair} in $P$ if\\ 
    \indent $\Wleft(u)$ is $x_0$-right of $\Wleft(v)$ and 
    $\Wright(u)$ is $x_0$-left of $\Wright(v)$;
\end{itemize}
Evidently, $(u,v)$ is an outside pair if and only if
$(v,u)$ is an inside pair.  Also, when $(u,v)\in\Inc(P)$, then
the pair $(u,v)$ is one of the four mutually exclusive types: left, right, 
outside, inside.  However, in contrast to the situation with left and right
pairs, it may happen that a comparable pair satisfies the requirements to be an inside pair 
or outside pair. 

\begin{example}\label{exa:4-types}
  For the poset whose cover graph is shown in Figure~\ref{fig:setup}:
  \begin{enumerate}
    \item Among the left pairs are: $(a,e)$, $(h,m)$, $(z,x)$ and $(p,t)$.
    \item Among the right pairs are: $(c,a)$, $(d,b)$, $(k,h)$, and $(t,q)$.
    \item Among the outside pairs are: $(m,b)$, $(s,v)$ and $(z,q)$.
    \item Among the inside pairs are: $(c,g)$, $(b,x)$, $(m,p)$ and $(q,z)$.
  \end{enumerate}
\end{example}

We introduce the following notation:
\begin{itemize}
  \item $X(\mathrm{left})$ denotes the set of all left pairs of $P$. 
  \item $X(\mathrm{right})$ denotes the set of all right pairs of $P$.
  \item $X(\mathrm{outside})$ denotes the set of all incomparable pairs of $P$ that are outside pairs.
\end{itemize}
Observe that all three sets $X(\mathrm{left})$, $X(\textrm{right})$, 
$X(\textrm{outside})$ are sets of incomparable
pairs in $P$.  

\begin{proposition}\label{pro:LO-RO} 
  The following two subsets of $\Inc(P)$ are reversible in $P$: 
  \begin{enumerate}
    \item  $X_1=X(\mathrm{left})\cup X(\mathrm{outside})$.
    \item  $X_2=X(\mathrm{right})\cup X(\mathrm{outside})$.
  \end{enumerate}
\end{proposition}
\begin{proof}
By Proposition~\ref{pro:Tleft-and-Tright}, the paths $\Wleft(u)$ form a
  rooted spanning tree.  Order its children by the local edge ordering and
  let $L_1$ be the preorder traversal that visits children from right to left.
  Thus, for distinct $a,b$, we have $a<b$ in $L_1$ precisely when
  $\Wleft(a)$ is a proper prefix of $\Wleft(b)$ or
  $\Wleft(b)$ is $x_0$-left of $\Wleft(a)$.
  If $a<_P b$, Proposition~\ref{pro:shortcuts} excludes the only remaining
  possibility, that $\Wleft(a)$ is $x_0$-left of $\Wleft(b)$.
  Hence $L_1$ extends $P$ and reverses every pair in $X_1$.

  Symmetrically, traverse the tree formed by the paths $\Wright(u)$ in
  preorder, visiting children from left to right.  The resulting linear
  extension $L_2$ reverses every pair in $X_2$.
  \end{proof}

Proposition~\ref{pro:LO-RO} implies
that it is essential that we understand more fully
the distinctive characteristics of inside pairs.

\subsection{Shadows, Shadow Depth and Shadow Sequences}

Let $W$ and $W'$ be a pair of witnessing paths.  
We will say that a set $\cgB$ of points in the plane is the \emph{non-degenerate block
enclosed by $W$ and $W'$} if the following three conditions are met: 
(1)~there are distinct elements $x$ and $y$ of $P$ such that $x<_P y$, $x$ is not
covered by $y$, and both $W$ and $W'$ are witnessing paths from $x$ to $y$;
(2)~$W\cap W'=\{x,y\}$ so $W\cup W'$ is a cycle $D$ in the cover graph $G$; and 
(3)~$\cgB$ is the set of points in the plane that are on or inside the closed curve determined by $D$.

On the other hand,
we will say that a set $\cgB$ of points in the plane is the \emph{degenerate block
enclosed by $W$ and $W'$} if the following three conditions are met: 
(1)~there are elements $x$ and $y$ of $P$ such that $x$ is covered
by $y$ in $P$; (2)~$W=W'$ is the $2$-element witnessing path from $x$ to
$y$ consisting of a single edge; and (3)~$\cgB$ is the set of points in the plane
on the edge in $G$ between $x$ and $y$.

We note that in both cases, the set $\cgB$ and the elements $x$ and $y$ are
determined by the pair $(W,W')$, so we will simply say that $\cgB$ is the \emph{block
enclosed by $W$ and $W'$}, thereby allowing $\cgB$ to be either degenerate or
non-degenerate.   Also, we will refer to the element $x$ as $\min(\cgB)$, while
the element $y$ will be $\max(\cgB)$.  When $\cgB$ is the block enclosed by 
$W$ and $W'$, it is important to note that if $w$ is an element of $P$ on the
boundary of $\cgB$, i.e., $w\in W\cup W'$, then $w\le \max(\cgB)$ in $P$.
Soon it will become clear that if $u$ is any element of $P$ that is the interior
of $\cgB$, then $\min(\cgB)< u$ in $P$.  However, in general, we will not know the order 
relation, if any, between $u$ and $\max(\cgB)$.

Let $\cgB$ be a non-degenerate block enclosed by a pair $(W,W')$ of witnessing paths from 
$\min(\cgB)$ to $\max(\cgB)$.  We will designate one of $W$ and $W'$ to be the 
\emph{left side of $\cgB$}, while the other
path will be the \emph{right side of $\cgB$}.  The distinction is made using
the following convention.  Consider a clockwise traversal of a portion of
the boundary of $\cgB$, starting at $\min(\cgB)$ and stopping at $\max(\cgB)$.  Then the
path we have followed is the left side of $\cgB$; the other side is
the right side of $\cgB$.  The elements $\min(\cgB)$ and $\max(\cgB)$ belong to
both sides.  An element $z$ on the left side of $\cgB$, with $\min(\cgB)< z< \max(\cgB)$ in $P$
is said to be \emph{strictly} on the left side of $\cgB$.  
An element $z$ on the right side of $\cgB$, with $\min(\cgB)< z< \max(\cgB)$ in $P$ 
is said to be \emph{strictly} on the right side of $\cgB$.  

When $\cgB$ is a degenerate block consisting of
a single edge, we consider this edge to be both the left side and the right side of 
$\cgB$, and there are no elements that are strictly on one of the two sides.

For the discussion to follow, we fix an element $z \in U_P(x_0)$.  Then
we will introduce notation that defines sequences, paths, blocks, and elements of
$P$, all depending on the choice of $z$.  However, to maintain some reasonable
level in the complexity of notation, we will not indicate the dependence on $z$.
With $z$ fixed, we let $W_L=\Wleft(z)$, $W_R=\Wright(z)$.  
Then $W_L\cap W_R$ is a chain containing
$x_0$ and $z$.  Accordingly, the elements of $W_L\cap W_R$ can be listed
sequentially as $(z_0,\dots,z_m)$ such that $i<j$ if $z_i<_P z_j$.  We will refer to
$(z_0,\dots,z_m)$ as the \emph{sequence of common points of $z$}.  Note that $m\ge1$,
$z_0=x_0$, and $z_m=z$.  For $x_0$ itself, the sequence of common points
is the one-term sequence $(x_0)$.  For all $i\in [m]$, 
we define $\cgB_i$ to be the block enclosed by 
$z_{i-1}W_Lz_i$ and $z_{i-1}W_Rz_i$. 
We also call $\cgB_i$ the block \emph{between $z_{i-1}$ and $z_i$}.
Each block in the sequence $(\cgB_1,\dots,\cgB_m)$ is either
degenerate or non-degenerate.  In general, any of the possible $2^m$ outcomes from
this binary distinction is possible.

We now make some elementary but nonetheless important observations about the
sequence $(\cgB_1,\dots,\cgB_m)$.  Suppose that $m\ge2$, and let $i$ be an integer with
$0<i< m$. 
Since $i<m$, we know $z_i\neq z$, and therefore $z_i<_P z$.
Since $w\le_P z_i$ for all elements $w$ of $P$ that are
on the boundary of $\cgB_i$, we know $z$ is not on the boundary of
$\cgB_i$.  We then have two cases:

\smallskip
\noindent
\textit{Case 1.}\quad  $z$ is in the exterior of $\cgB_i$.

In this case, it follows that if $W$ is a witnessing path from $z_i$ to $z$, then all 
edges and vertices of $W$, except the starting point $z_i$ are in the exterior
of $\cgB_i$. 
We apply this reasoning twice, first when $W=z_iW_Lz$, and second
when $W=z_iW_Rz$. 
In particular, we conclude that the boundaries of $\cgB_i$ and $\cgB_{i+1}$ stay disjoint, apart from $z_i$. 
Note also that $\cgB_{i+1}$ cannot contain $\cgB_i$ as 
in this case every witnessing path from
$x_0$ (recall that $x_0$ lies on the exterior face) to $z_{i-1}$ would intersect the boundary of $\cgB_{i+1}$ but all the elements 
at the boundary of $\cgB_{i+1}$ are 
larger than or equal to $z_i$ in $P$.
We conclude that (1)~$z_i$ is the only point in the plane common to 
$\cgB_i$ and $\cgB_{i+1}$; and (2)~if $i+1<j\le m$, then $\cgB_i$ and $\cgB_j$ are disjoint.

\smallskip
\noindent
\textit{Case 2.}\quad  $z$ is in the interior of $\cgB_i$.

In this case, it follows that if $W$ is a witnessing path from $z_i$ to $z$, all 
edges and vertices of $W$, except the starting point $z_i$ are in the interior
of $\cgB_i$.  Again, we apply this reasoning twice, first when $W=z_iW_Lz$, and second
when $W=z_iW_Rz$.  Now we conclude that (1)~$z_i$ is the only point in the plane common to 
the boundaries of $\cgB_i$ and $\cgB_{i+1}$; (2)~all points in the plane belonging
to $\cgB_{i+1}$, except $z_i$, are in the interior of $\cgB_i$; and (3)~if $i+1<j\le m$, 
then all points of $\cgB_j$ are in the interior of $\cgB_i$.

When this second case holds, and $z$ is in the interior of $\cgB_i$, we will call the
element $z_i$ a \emph{reversing element} for $z$.  The 
number of reversing elements in the sequence of common elements
of $z$ will be called the \emph{shadow depth of $z$}, denoted $\sd(z)$. 
By convention, $\sd(x_0)=-1$; no shadows of $x_0$ are defined.

Next, when $\sd(z)=r$, we define a sequence $(\shad_0(z),\dots,\shad_r(z))$
of sets called the \emph{shadow sequence of $z$}.
First, let $(i_1,\dots,i_r)$ be the subsequence of $(1,\dots,m-1)$ determined by
the subscripts of the reversing elements of $z$.
We expand this sequence by
adding $i_0=0$ at the beginning, and then adding $i_{r+1}=m$ at the end.
For each $j$ with $0\le j\le r$, we then set
\[
  \shad_j(z) = \bigcup_{\ \ \, i_{j}\, <\, i\, \le\, i_{j+1}} \cgB_i.
\] 
Also, we refer to $(\cgB_{i_{j}+1},\ldots, \cgB_{i_{j+1}})$ as the \emph{sequence of blocks} of 
the $j$-shadow of $z$.
We call $\cgB_{i_{j+1}}$ the \emph{terminal block} of $\shad_j(z)$.
Also, we call the element $z_{i_{j}}$ the \emph{base element} of $\shad_j(z)$,
and we call $z_{i_{j+1}}$ the \emph{terminal element} of $\shad_j(z)$.  These
definitions are illustrated in Figure~\ref{fig:shadows}.  Note that this
is our familiar poset from Figure~\ref{fig:setup}, with certain elements relabeled.

\begin{figure}[!h]
  \begin{center}
    \includegraphics[scale=.8]{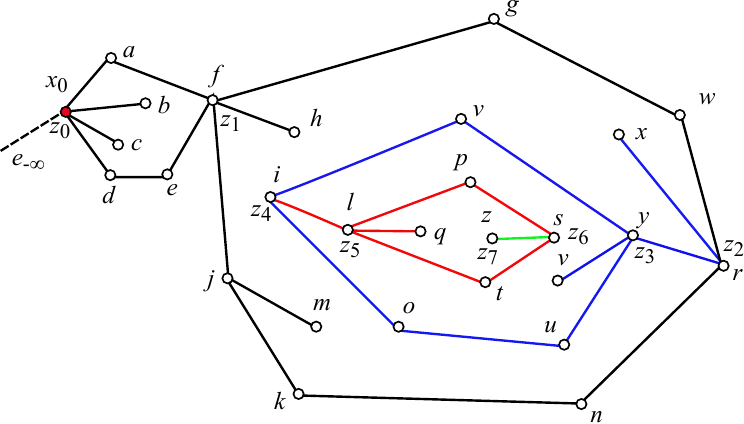}
  \end{center}
  \caption{
    The sequence of common points of $z=z_7$ is $(z_0,\dots,z_7)$.  
    The shadow depth of $z$ is $3$, and its reversing elements are $z_2$, $z_4$, 
    and $z_6$.  Shadow blocks $\cgB_3,\cgB_5,\cgB_7$ are degenerate, while 
    $\cgB_1,\cgB_2,\cgB_4,\cgB_6$ are non-degenerate.  
    In this figure, 
    the black and red edges are oriented left-to-right in the plane, 
    while the blue and green edges are oriented right-to-left.
  } 
  \label{fig:shadows}
\end{figure}

We have already made some basic observations
concerning shadows, shadow sequences and shadow depth.  Now we give a series
of four additional propositions concerning these concepts.  The first of these includes
only some nearly self-evident statements listed explicitly for emphasis, so
no additional arguments are given.  However, three more substantive propositions follow.
\begin{proposition}\label{pro:shadows-basic-props} 
  Let $z\in U_P(x_0)$, let $(z_0,\ldots,z_m)$ be the sequence of common points of $z$.
  If $i\in[m]$, then the sequence of common elements of $z_i$ is $(z_0,\dots,z_i)$.
  Furthermore, if $j$ is an integer with $0\le j\le\sd(z)$, and $\cgB$ is a block
  of $\shad_j(z)$, then the following statements hold:
  \begin{enumerate} 
    \item If $x$ is the base element of $\shad_j(z)$, then $\sd(x)=j-1$, unless $j=0$, in which case we have $x=x_0$.
    \item $\shad_j(\max(\cgB))\subseteq\shad_j(z)$, with equality holding if and only 
      if $\cgB$ is the terminal block of $\shad_j(z)$.  
  \end{enumerate}
\end{proposition}
Let $\cgB$ be a block.  We will refer to $\cgB$ as a \emph{shadow
block} when there is an element $z\in U_P(x_0)$ such that if $(z_0,\dots,z_m)$
is the sequence of common points of $z$, then there is some $i\in[m]$ such
that $\cgB$ is the block between $z_{i-1}$ and $z_i$.
\begin{proposition}\label{pro:uWLWR}
  Let $\cgB$ be a shadow block and 
  let $u$ be an element of $P$ that belongs to $\cgB$. 
  Then the following statements hold.
  \begin{enumerate}
    \item\label{pro:uWLWR:item:minB} $\min(\cgB)$ belongs to both $\Wleft(u)$ and $\Wright(u)$;
    \item\label{pro:uWLWR:item:common-points} the sequence of common points of $\min(\cgB)$ is a prefix of 
    the sequence of common points of $u$;
    \item\label{pro:uWLWR:item:shadows-lower} $\shad_j(\min(\cgB)) = \shad_j(u)$ for all $0\leq j < \sd(\min(\cgB))$; and
    \item\label{pro:uWLWR:item:shadow-equal} $\shad_j(\min(\cgB)) \subseteq \shad_j(u)$ if $0\le j=\sd(\min(\cgB))$.
  \end{enumerate}
\end{proposition}

\begin{proof}
  Note that items~\ref{pro:uWLWR:item:common-points}-\ref{pro:uWLWR:item:shadow-equal} follow immediately from item~\ref{pro:uWLWR:item:minB} and the fact that leftmost (rightmost) witnessing paths are $x_0$-consistent, see~Proposition~\ref{pro:Tleft-and-Tright}.

  To simplify the notation for the argument, we set 
  $x=\min(\cgB)$, and $y=\max(\cgB)$.  Our goal is to show that $x$ belongs to
  both $\Wleft(u)$ and $\Wright(u)$.  We provide full details to show
  that $x$ belongs to $\Wleft(u)$.  The argument for $\Wright(u)$ is symmetric.

  Since $x_0$ is in the exterior face and $u$ is in $\cgB$, 
  the path $\Wleft(u)$ intersects the boundary of $\cgB$. 
  Let $w$ be the least element of $P$ that is on $\Wleft(u)$ and is in
  $\cgB$. In particular, $\Wleft(w)$ is a prefix of $\Wleft(u)$ and
  $w$ is on the boundary of $\cgB$.
  If $w$ is on the left side of 
  $\cgB$, i.e.\ $w\in x\Wleft(y)y$, 
  then $\Wleft(w)$ is a prefix of $\Wleft(y)$ and $\Wleft(x)$ is a prefix of $\Wleft(w)$. 
  Thus $x\in\Wleft(u)$, as desired.

  We may therefore assume that $w$ is strictly on the right side of $\cgB$ and 
  argue for contradiction. 
  In particular, we have $x<w< y$ in $P$.

  Now consider the $x_0$-consistent paths $\Wleft(x)$ and $\Wleft(w)$.
  Suppose for the sake of contradiction that $x\not\in W_L(u)$. 
  Then $W_L(x)$ is not a prefix of $W_L(w)$, and since $x < w$, 
  we clearly have that $W_L(w)$ is not a prefix of $W_L(x)$. 
  Thus we have that one of $W_L(w)$, $W_L(x)$ is $x_0$-left of the other. 
  The comparability $x< w$ in $P$ and 
  Proposition~\ref{pro:shortcuts} imply that $\Wleft(w)$ is
  $x_0$-left of $\Wleft(x)$.  
  Let $v$ be the largest element of
  $P$ common to $\Wleft(w)$ and $\Wleft(x)$. 
  Let $e_0$ be the edge common to $\Wleft(w)$ and $\Wleft(x)$ that is immediately before
  $v$, if $v=x_0$ then $e_0=e_{-\infty}$. 
  Then let $e_w=vw'$ be the edge on $\Wleft(w)$ that is immediately after
  $v$, and let $e_x=vx'$ be the edge of $\Wleft(x)$ that is immediately after $v$.
  Then $e_w$ is left of $e_x$ in the $(v,e_0)$-ordering.

  Recall that since $\cgB$ is a shadow block
  $\Wleft(x)$ is a prefix of $\Wleft(y)$. 
  Therefore, $v < x' \leq x < y$ in $P$ and all these elements lie on $\Wleft(y)$. 
  However, $v < w < y$ in $P$ as well and therefore 
  the construction procedure of $\Wleft(y)$ after reaching $v$ should prefer $e_w$ over $e_x$. 
  This contradicts the fact that $e_x$ is on $\Wleft(y)$ 
  (since it is on $\Wleft(x)$, which is a prefix of $\Wleft(y)$).
  \end{proof}

\begin{proposition}\label{pro:path-in-block}
  Let $\cgB$ be a shadow block.  If $W$ is a witnessing path in $P$,
  and both the starting and the ending point of $W$ are in $\cgB$,
  then all edges of $W$ are in $\cgB$. 
\end{proposition}
\begin{proof}
  If $\cgB$ is degenerate, its endpoints form a cover relation, so the
  statement is immediate.  Assume that $\cgB$ is non-degenerate. 
  We argue by contradiction.  Among all witnessing paths for which the
  proposition fails, 
  we take $W$ as one of minimum length.  Let $u$ and $v$ be the starting and
  ending points, respectively, of $W$.  Then $u$ and $v$ are on the boundary of
  $\cgB$, while all edges of $W$, and any vertices of $W$ that
  are between $u$ and $v$ are in the exterior of $\cgB$.
  The notation for the remainder of the argument is simplified by
  setting $x=\min(\cgB)$, and $y=\max(\cgB)$.
  Then $x\le u< v\le y$ in $P$.  So $u\neq y$ and $v\neq x$.
  Let $N$ be the uniquely determined (and possibly trivial) witnessing
  path from $v$ to $y$ that is a suffix of one of the two sides of
  $\cgB$.  Let $e$ be the first edge of $W$.

  Since $\cgB$ is a shadow block, the left side of $\cgB$ is contained in 
  $\Wleft(y)$ and the right side of $\cgB$ is contained in $\Wright(y)$. 
  If $u$ is on the left side of $\cgB$, let
  $e^+_L$ and $e^-_L$ be, respectively, the edges immediately after and
  immediately before $u$ on the path $\Wleft(y)$. 
  Of course, in the case $u=x_0$ we set $e^-_L=e_{-\infty}$. 
  Also, if $u$ is on the
  right side of $\cgB$, let $e^+_R$ and
  $e^-_R$ be, respectively, the edges immediately after and immediately
  before $u$ on $\Wright(y)$. 
  Again, if $u=x_0$ we set $e^-_R=e_{-\infty}$. 
  Note that if $u=x$, then $u$ is on both
  sides of $\cgB$. 

  If $u$ is strictly on the left side of $\cgB$, then the fact that
  $e$ is in the exterior of $\cgB$ implies that $e$ is left of $e^+_L$ in
  the $(u,e^-_L)$-ordering.  Let $N'$ be the concatenation of 
  $\Wleft(u)$, $W$ and $N$.  Then $N'$ is a witnessing path from
  $x_0$ to $y$ contradicting the fact that
  $\Wleft(y)$ is leftmost.  We conclude that $u$ cannot be
  strictly on the left side of $\cgB$. A symmetric argument shows
  that $u$ cannot be strictly on the right side of $\cgB$.  We conclude
  that $u=x$.

  We now identify two cases according to the cyclic ordering on 
  the edges $e^+_L,e^-_L,e^+_R,e^-_R$. 

  \smallskip\noindent
  \textit{Case 1.}\quad $e^-_L\prec e^+_L\prec e^+_R\prec e^-_R \preceq e^-_L$.

  We note that this case holds when $x=x_0$. 
  Now we claim that $e^-_L\prec e\prec e^-_R$. 
  This claim holds trivially if $x=x_0$. 
  Now suppose $x\neq x_0$ and let $(z_0,\ldots,z_m)$ be the sequence 
  of common points of $y$. 
  The assumption in Case 1 holds when $m\ge2$ and
  $x=z_{m-1}$ is \emph{not} a reversing element for $y$.  
  Then $y$ is in the exterior of $\cgB_{m-1}$---the block between $z_{m-2}$ and $z_{m-1}$. 
  As noted previously,
  this requires that all vertices and edges of witnessing paths from $x$ to
  $y$, except the vertex $x$, are in the exterior of $\cgB_{m-1}$.  
  This requires the cyclic ordering $e^-_L\prec e \prec e^-_R$, as desired.

  Since $e$ is not in $\cgB$, it follows that either $e^-_L\prec e\prec e^+_L$,
  or $e^+_R\prec e\prec e^-_R$.  If the first option holds, then
  the concatenation of $\Wleft(x)$, $W$ and $N$ contradicts $\Wleft(y)$ being
  leftmost.  If the second option holds, then the concatenation of $\Wright(x)$,
  $W$ and $N$ contradicts $\Wright(y)$ being rightmost.  With these observations,
  the proof of Case~1 is complete.

  \smallskip\noindent
  \textit{Case 2.}\quad $e^-_L\prec e^-_R\prec e^+_L\prec e^+_R\prec e^-_L$.

  We note that in this case $x\neq x_0$. 
  Let $(z_0,\ldots,z_m)$ be the sequence 
  of common points of $y$. 
  We have $m\ge2$, and $x=z_{m-1}$ is a reversing
  element of $y$. Thus, $y$ is in the interior of $\cgB_{m-1}$.
  Also, all edges and vertices of witnessing paths from $x$ to $y$, 
  except $x$, are
  in the interior of $\cgB_{m-1}$. 
  It follows that we have the cyclic ordering
  $e^-_R\prec e\prec e^-_L$.  Since $e$ is not in $\cgB$, it follows that either
  $e^-_R\prec e\prec e^+_L$ or $e^+_R\prec e\prec e^-_L$.  If the first
  option holds, then the concatenation of $\Wleft(x)$, $W$ and $N$ contradicts
  $\Wleft(y)$ being leftmost.  If the second option holds, then
  the concatenation of $\Wright(x)$, $W$ and $N$ contradicts $\Wright(y)$ being
  rightmost.  Since the two cases are exhaustive, the proof of the proposition is complete.
\end{proof}

\begin{proposition}
\label{pro:shadow-block-maxB}
  Let $\cgB$ be a shadow block and 
  let $u$ be an element of $P$ that belongs to $\cgB$. 
  If $u> \max(\cgB)$, then $\max(\cgB)$ belongs to both $\Wleft(u)$ and $\Wright(u)$.
\end{proposition}
\begin{proof}

    Set $x=\min\cgB$, and $y=\max\cgB$.    By~Proposition~\ref{pro:uWLWR}.\ref{pro:uWLWR:item:minB}, $x$ lies in $W_L(u)$ and $W_L(y)$. Let $W$ be the leftmost witnessing path from $y$ to $u$.
Consider the paths $x W_L(y) y$ and $x W_L(u) u$ and $yWu$. All of these paths lie in $\cgB$ by Proposition~\ref{pro:path-in-block}. So the witnessing path $x W_L(y) y W u$ is in $\cgB$. 

If the witnessing paths $x_0 W_L(u) u$ and $x_0 W_L(y) y W u$  are the same, then $y$ belongs to $W_L(u)$ as desired. So suppose instead that these paths differ.
Let $e=zz'$ be the first edge of $x_0 W_L(y) y W u$ which is not on $x_0 W_L(u) u$. As we have seen, $x \leq z <z'$.  Let $e^-$ be the edge that precedes $e$ on $x_0 W_L(y) y W u$. (If $z = x_0$, then let $e^- = e_{-\infty}$.) So $e^-$ is also on $x_0 W_L(u) u$. Let $e'$ be the edge following $z$ on $x_0 W_L(u) u$. 
By the definition of $W_L(u)$, $e'$ must be left of $e$ in the $(z,e^-)$-ordering.
Suppose $z$ is on $x W_L(y) y$. Then $e$ is outside $\cgB$, which is a contradiction. 
Therefore, $w$ is on $y W u$. But this implies $y$ belongs to $W_L(u)$ as desired. The proof that $y$ belongs to $W_R(u)$ is symmetric. 
\end{proof}

\begin{proposition}\label{pro:shadow-comp}
  Let $z\in U_P(x_0)$, let $j$ be an integer with $0\le j\le\sd(z)$, let $\cgB$ be a block
  of $\shad_j(z)$, let $y=\max(\cgB)$, and let $u$ be an element of 
  $P$ that belongs to $\cgB$ but is not the base element of $\shad_j(z)$.  
  Then the following statements hold:
  \begin{enumerate}
    \item
    \label{pro:item:depth}
      $\sd(u)\ge j$;
    \item
    \label{pro:item:depth-with-y}
      $\shad_j(u)\subseteq \shad_j(y)$, with equality holding if
      and only if $u\ge_P y$;
    \item
    \label{pro:item:depth-of-u}
    if $u>y$ in $P$, then $\cgB$ is the terminal block of $\shad_j(u)$ and $\sd(u)\geq j+1$.
  \end{enumerate}
\end{proposition}
\begin{proof}
  We first prove statement~\ref{pro:item:depth}. 
  Let $(z_0,\ldots,z_m)$ be the sequence of common points of $z$. 
  Let $x=\min{\cgB}$.  
  Note that $x$ is in the sequence, say $x=z_i$. 
  Let $z_{i'}$ be the base element of the $\shad_j(z)$. 
  By our assumption we have $z_{i'} \leq z_i=x \leq u$ in $P$ and at least 
  one of the inequalities is strict.
  Proposition~\ref{pro:uWLWR} implies
  that $x$ is also a common point for $u$. 
  Thus, $(z_0,\ldots,z_i)$ is a prefix of the sequence of common points of $u$.
  Of these, $j$ elements are reversing for $z$ so also for $u$
  (and in the case $z_i=u$, the $j$ reversing elements are before $z_i$ as 
  $z_{i'}<u$ in $P$).
  Thus, $\sd(u)\ge j$.  This completes the
  proof of statement~\ref{pro:item:depth}.

  Now we prove statement~\ref{pro:item:depth-with-y}.  Any block of $\shad_j(z)$ that precedes $\cgB$ is
  also a block of both $\shad_j(u)$ and $\shad_j(y)$.  Now consider
  a block $\cgD$ of $\shad_j(u)$ with $x\le\min(\cgD)$.  Then the
  two sides of $\cgD$ are witnessing paths from $\min(\cgD)$ to
  $\max(\cgD)$.  However, $x\le \min(\cgD)<\max(\cgD)\le u$ in $P$.  It follows
  from Proposition~\ref{pro:path-in-block}, applied to $x\Wleft(u)u$
  and $x\Wright(u)u$, that all edges and vertices of the
  two sides of $\cgD$ are in $\cgB$.  This implies that $\cgD\subseteq\cgB\subseteq
  \shad_j(y)$. 

  Now suppose that $\shad_j(y)=\shad_j(u)$. 
  Recall that all elements $w$ of $P$ on the boundary of $\shad_j(u)$ 
  satisfy $w\leq u$ in $P$. Since $y$ is on the boundary of $\shad_j(u)$, 
  we conclude $y \leq u$ in $P$. 
  Conversely, assume that $u\geq y$ in $P$. 
  If $u=y$, then there is nothing to prove. 
  If $u> y$ in $P$, then
  Proposition~\ref{pro:shadow-block-maxB} puts $y$ on both extreme paths
  to $u$.  Since $u$ lies in the interior of $\cgB$, $y$ is a reversing
  element for $u$.  Consequently $\shad_j(u)=\shad_j(y)$,
  $\cgB$ is its terminal block, and $\sd(u)\ge j+1$.
  This proves both the reverse implication in statement~\ref{pro:item:depth-with-y} and statement~\ref{pro:item:depth-of-u}.
\end{proof}

\begin{corollary}
\label{cor:b-and-c-have-the-same-shadow}
    Let $b,c\in U_P(x_0)$ with $b<c$ in $P$. Let $j$ be an integer with $0\leq j<\sd(b)$. 
    Then
    \[\shad_j(b)=\shad_j(c).\]
\end{corollary}
\begin{proof}
Let $\cgB$ be the terminal block of $\shad_j(b)$. 
Thus, $b$ lies in the interior of $\cgB$. 
Consider a witnessing path from $b$ to $c$ in $P$.
Since for every element $z$ of $P$ at the boundary of $\cgB$, we have $z<b$ in $P$, 
we conclude that the witnessing path is contained in the interior of $\cgB$. 
In particular, $c$ lies in the interior of $\cgB$. 
Now applying~Proposition~\ref{pro:shadow-comp}.\ref{pro:item:depth-with-y} twice for $b$ and $c$, we obtain that 
$\shad_j(b) =\shad_j(\max(\cgB))=\shad_j(c)$.
\end{proof}

We will also use the following consequence of the common-prefix property.
If $\cgB$ is a shadow block, $j=\sd(\max(\cgB))$, and
$u\in\cgB\setminus\{\min(\cgB)\}$, then
\[
  \shad_i(u)=\shad_i(\max(\cgB))\qquad(0\le i<j).
\]
Indeed, Proposition~\ref{pro:uWLWR} gives the common prefix through
$x=\min(\cgB)$.  If $x$ is not the base element of the $j$-shadow,
then $\sd(x)=j$, so Proposition~\ref{pro:uWLWR} applies directly.
If $x$ is the base and $j>0$, then $\sd(x)=j-1$ and $x$ is reversing
for both $u$ and $\max(\cgB)$: the block $\cgB$, apart from $x$, lies
inside the preceding terminal block.  Their $(j-1)$-shadows therefore
also agree.  For $j=0$ the assertion is vacuous.

\subsection{The Address of an Inside Pair}

Let $(a,b)\in\Inc(P)$. 
We define the \emph{depth of $(a,b)$} to be the least integer
$j\in\{0,\ldots,\min(\sd(a),\sd(b))\}$ such that
$\shad_j(a)\neq\shad_j(b)$. 
Such an index exists: otherwise, assuming
$\sd(a)\le\sd(b)$, their common $\sd(a)$-shadow would have terminal
element $a$, giving $a\le_P b$, a contradiction. 
Note that 
if $j$ is the depth of $(a,b)$, then (1)~$\shad_j(a)$ and
$\shad_j(b)$ have the same base element; and (2)~$\shad_j(a)$ and $\shad_j(b)$
do not have the same terminal element.

\begin{proposition}\label{pro:geometry-inside}
  Let $(a,b)\in\Inc(P)$, and let $j$ be the depth of $(a,b)$.  Then
  $(a,b)$ is an inside pair if and only if $a$ is in the 
    interior of a block of $\shad_j(b)$.
\end{proposition}
\begin{proof}
  Suppose first that $a$ is in the interior of the block $\cgB$ of $\shad_j(b)$.
  We show that $(a,b)$ is an inside pair, i.e., we must show (1)~$\Wleft(b)$ is
  $x_0$-left of $\Wleft(a)$; and (2)~$\Wright(b)$ is $x_0$-right of $\Wright(a)$.
  We prove the first of these two statements.  
  The argument for the second is symmetric. 
  Let $x=\min(\cgB)$ and let $y=\max(\cgB)$. 
  Proposition~\ref{pro:uWLWR} implies that $x$ is in both 
  $\Wleft(a)$ and $\Wright(a)$. 
  This means that $\Wleft(a)$ and $\Wleft(b)$ coincide from $x_0$ to $x$, 
  and also $\Wright(a)$ and $\Wright(b)$ coincide from $x_0$ to $x$.

  Note that if $a\geq y$ in $P$, then 
  Proposition~\ref{pro:shadow-comp} implies $\shad_j(a)=\shad_j(y)$. 
  Since the depth of $(a,b)$ is $j$, we must have 
  $\shad_j(a)=\shad_j(y)\subsetneq \shad_j(b)$. 
  This means that if $a\geq y$ in $P$, then 
  $b$ is not in $\cgB$ and the edges of $\Wleft(b)$ and $\Wright(b)$ 
  immediately after $y$ are not in $\cgB$.

  Since $x$ and $a$ are in $\cgB$, 
  Proposition~\ref{pro:path-in-block} implies that all edges and vertices
  of $x\Wleft(a)a$ are in $\cgB$.  
  Let $z$ be the largest point
  on the left side of $\cgB$ such that $z\in\Wleft(a)$. 
  As we discussed $x\leq z\leq y$ in $P$ and clearly, 
  $x\Wleft(a)z$ is a prefix of the left side of $\cgB$. 
  Let $e^+$ and $e^-$ be the edges of
  $\Wleft(b)$ that are, respectively, immediately after and immediately before
  the vertex $z$.  
  Again, if $z=x=x_0$ then $e^-=e_{-\infty}$. 
  Let $e$ be the edge of $\Wleft(a)$ immediately after
  $z$.  
  Since $e$ is inside $\cgB$ and 
  $e^+$ is on the boundary of $\cgB$ or outside $\cgB$ 
  (if $z=y$; by the previous paragraph), 
  we conclude that 
  $e^+$ is left of $e$ in the $(z,e^-)$-ordering. 
  This implies $\Wleft(b)$ is
  $x_0$-left of $\Wleft(a)$. 
  A symmetric argument shows
  that $\Wright(b)$ is $x_0$-right of $\Wright(a)$.  Together, these statements
  imply that $(a,b)$ is an inside pair.

  For the second part of the proof, we assume that $(a,b)$ is an inside
  pair and show that there is some block of $\shad_j(b)$ such that
  $a$ is in the interior of that block.  Let $(a_0,\dots,a_m)$ and
  $(b_0,\dots,b_n)$ be the sequences of common elements of $a$ and $b$,
  respectively. 
  Since $(a,b)$ is of depth $j$, 
  $\shad_j(a)$ and $\shad_j(b)$ have the same base element, 
  which we denote as $x_j$.
  Let $x$ be the last element of the common prefix of the two 
  sequences, say $x=a_s=b_s$. 
  Note that $x$ does not occur before $x_j$ and $x\not\in\set{a,b}$.
  Let $y=b_{s+1}$ be the common point of $b$ that occurs in the sequence immediately after $x$ and let $\cgB$ be the block of $\shad_j(b)$ with $x=\min(\cgB)$ and $y=\max(\cgB)$.
Let $e^-_L$, $e^-_R$ be the edges of $\Wleft(b)$ and $\Wright(b)$ respectively that are immediately before $x$ 
(as usual $e^-_L=e^-_R=e_{-\infty}$ if $x=x_0$). 
Note that these edges are also in $\Wleft(a)$ and $\Wright(a)$, respectively. 
Let $e^+_L$, $e^+_R$ be the edges of $\Wleft(b)$ and $\Wright(b)$, respectively, that are immediately after $x$.

We first consider the case when $x=x_j$.
In this case, the clockwise cyclic ordering of our four distinguished edges 
around $x$ is 
$e^-_L \preceq e^-_R \prec e^+_L \preceq e^+_R \prec e^-_L$ 
(the first $\preceq$ becomes $=$ only if $x=x_j=x_0$ and then both edges coincide with $e_{-\infty}$).
Now immediately from the fact that $(a,b)$ is an inside pair, we conclude that 
$e^+_L \preceq e \preceq e^+_R$
when $e$ is either 
the first edge of $x\Wleft(a)a$  or 
the first edge of  $x\Wright(a)a$.
Therefore both edges belong to $\cgB$. 

  We claim that $a$ is in $\cgB$. 
  In order to prove that, note first that at least one of the paths 
  $x\Wleft(a)a$, $x\Wright(a)a$ is not going through $y$. 
  Indeed, if both contain $y$, 
  then $y$ would be a common point of $a$, but it is not.
  Let $W$ be one of the two paths $x\Wleft(a)a$, $x\Wright(a)a$ 
  avoiding $y$. 
  We already proved that the first edge of $W$ is in $\cgB$.
  Now in order to get a contradiction suppose that $a$ is not in $\cgB$. 
  Therefore, $W$ has to eventually leave $\cgB$. 
  Let $z$ be the first element of $W$ such that the edge $e$ immediately after $z$ on $W$ is not in $\cgB$. 
  This implies that $z$ is on the boundary of $\cgB$. 
  Since $z\not\in\set{x,y}$, we have $x<z<y$ in $P$. 
  Thus, $z$ is strictly on one of the sides of $\cgB$. 
  If $z$ is on the left side, then let $e^-$ and $e^+$ be the edges 
  of $\Wleft(b)$, respectively, immediately before and immediately after $z$. 
  Since $e$ is not in $\cgB$, we have that $e$ is left of $e^+$ 
  in the $(z,e^-)$-ordering. 
  Now Proposition~\ref{pro:piotrek} implies that 
  $\Wleft(a)$ is $x_0$-left of $\Wleft(b)$, 
  contradicting that $(a,b)$ is an inside pair. 
  If $z$ is on the right side, then we argue symmetrically and 
  obtain that $\Wright(a)$ is $x_0$-right of $\Wright(b)$, 
  which is also a contradiction.

  Now we proceed with the case $x\neq x_j$ 
  so $x$ is not a reversing element of $b$. 
  In this case, the clockwise cyclic ordering of our four distinguished edges 
around $x$ is 
$e^-_L \prec e^+_L \preceq e^+_R \prec e^-_R \preceq e^-_L$.
Again, immediately from the fact that $(a,b)$ is an inside pair, we conclude that either $e^-_R \prec e \prec e^-_L$ or 
$e^+_L \preceq e \preceq e^+_R$
when $e$ is either 
the first edge of $x\Wleft(a)a$  or 
the first edge of  $x\Wright(a)a$. 
Consider first the case that $e^-_R \prec e \prec e^-_L$. 
Let $x'$ be the common point of $b$ (and $a$) just before $x$ in the sequence of the common points. Let $\cgB'$ be the block of $\shad_j(b)$ (and $\shad_j(a)$) between $x'$ and $x$. In this case, we have that $x$ is a reversing element of $a$ and $a$ is in $\cgB'$, as desired.
Now consider the case that $e^+_L \preceq e \preceq e^+_R$. 
This means that the first edges of both $x\Wleft(a)a$ and $x\Wright(a)a$ are both in $\cgB$.
This is exactly the setup of the proof we had when $x=x_j$. 
Thus, we proceed as before and we prove that $a$ has to be in $\cgB$, as desired.
This completes the proof of the proposition.
\end{proof}

Using Proposition~\ref{pro:geometry-inside},
we make the following definition. The \emph{address} of an inside pair $(a,b)$ is the
uniquely determined pair $(j,\cgB)$ 
such that 
\begin{enumerate}
  \item $j$ is the depth of $(a,b)$; and 
  \item $\cgB$ is the block of $\shad_j(b)$
    that contains $a$ in its interior.  
\end{enumerate}

Now let $(a,b)$ be an inside pair, let $(j,\cgB)$ be the address of $(a,b)$,
and let $y=\max(\cgB)$.  Then $j=\sd(y)$, and $y\le b$ in $P$.  This
implies $a\not\le y$ in $P$.
If $\cgB$ is the terminal block of $\shad_j(b)$, then $a\parallel y$ in $P$ 
(as otherwise $a> y$ in $P$ and this implies by Proposition~\ref{pro:shadow-comp}.\ref{pro:item:depth-of-u} that $\shad_j(a)=\shad_j(y)=\shad_j(b)$, which is false).
If $\cgB$ is not the terminal block of 
$\shad_j(b)$, then either $a\parallel y$ in $P$ or $a> y$ in $P$.
\begin{example}\label{exa:addresses}
  We illustrate the notion of addresses using the poset whose cover
  graph is shown in Figure~\ref{fig:shadows}:
  \begin{enumerate}
    \item The address of $(b,z)$ is $(0,\cgB_1)$.
    \item The address of $(h,z)$ is $(0,\cgB_2)$.
    \item The address of $(v,z)$ is $(1,\cgB_4)$.
    \item The address of $(q,z)$ is $(2,\cgB_6)$.
  \end{enumerate}
\end{example}

In an effort to avoid the pathology displayed by the pairs in
Figure~\ref{fig:pathology}, we introduce a notion of parity for
inside pairs.  For each $\theta\in\{0,1\}$, we let $I_\theta$ consist of all inside
pairs $(a,b)$ such that if $(j,\cgB)$ is the address of $(a,b)$,
then $j\equiv\theta\mod2$. The set of all inside pairs is
partitioned as $I_0 \sqcup I_1$.  

\begin{figure}[!h]
  \begin{center}
    \includegraphics[scale=.8]{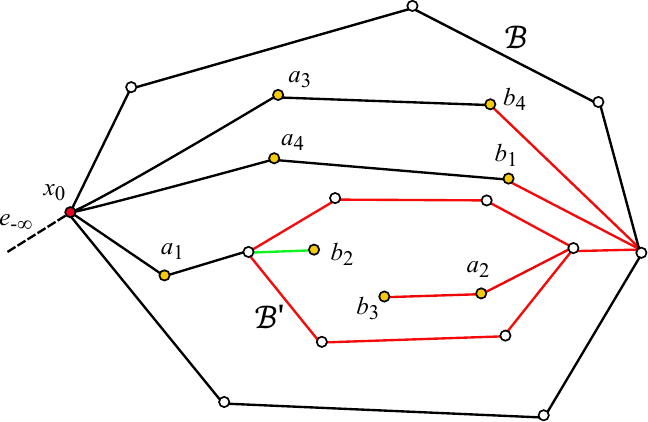}

  \end{center}
  \caption{$((a_1,b_1),\dots,(a_4,b_4))$ is a strict alternating cycle of inside
  pairs.  The common address of $(a_1,b_1)$, $(a_3,b_3)$, and $(a_4,b_4)$ is $(0,\cgB)$.
  However, the address of $(a_2,b_2)$ is $(1,\cgB')$. 
  These alternating cycles with pairs of different addresses is what we want to avoid 
  by partitioning $I$ into $I_0$ and $I_1$. 
  In this figure, 
    the black and green edges are oriented left-to-right in the plane, 
    while the red edges are oriented right-to-left.
  }
  \label{fig:pathology}
\end{figure}

We now prove a comprehensive technical lemma that provides 
structural information about strict alternating cycles 
in $I_0$ and $I_1$.  The conclusions of this lemma justify the notion of parity
for inside pairs.
\begin{lemma}\label{lem:cgI-comprehensive}
  Let $\theta\in\{0,1\}$, and let $((a_1,b_1),\dots,(a_k,b_k))$ be a strict
  alternating cycle in $P$ of pairs from $I_\theta$. 
  Then there is a pair $(j,\cgB)$ that is 
  a common address of $(a_{\alpha},b_{\alpha})$ for all $\alpha\in[k]$. 
  Moreover, for all $\alpha\in[k]$
  the following statements hold:
  \begin{enumerate}
    \item $\cgB$ is the terminal block of $\shad_j(b_\alpha)$;
    \item $a_\alpha$ and $b_\alpha$ are in the interior of $\cgB$; 
    \item $a_\alpha\parallel\max(\cgB)$ and $b_\alpha>\max(\cgB)$ in $P$;
    \item $\sd(a_\alpha)=j$, $\sd(b_{\alpha})>j$.\label{lem:comprehensive:item:sd}
  \end{enumerate}
\end{lemma}
\begin{proof}
  For each $\alpha\in[k]$, let $j_\alpha$ be the depth of the
  pair $(a_\alpha,b_\alpha)$.  Set $j=\min\{j_\alpha:\alpha\in[k]\}$.  
 For each $\alpha\in[k]$, choose a block $\cgB_\alpha$ of
  $\shad_j(b_\alpha)$ containing $a_\alpha$ in its interior, as follows.
  If $j_\alpha=j$, use Proposition~\ref{pro:geometry-inside}.
  If $j_\alpha>j$, then
  $\shad_j(a_\alpha)=\shad_j(b_\alpha)$ and both elements have shadow
  depth greater than $j$.  Both therefore lie in the interior of the
  terminal block of this common shadow; choose that block as $\cgB_\alpha$.

  The remainder of the argument will be organized within a series of three claims.

  \begin{claim}\label{clm:cgI:1}
    There is a block $\cgB$ such that for all $\alpha\in[k]$, the following 
    statements hold: 
    \begin{enumerate}
      \item $\cgB=\cgB_\alpha$ is the terminal block of $\shad_j(b_\alpha)$.
      \item $a_\alpha$ and $b_\alpha$ are in the interior of $\cgB$. 
      \item  $b_\alpha>\max(\cgB)$ in $P$.  
    \end{enumerate}
  \end{claim}
  \begin{proof}
    Let $\alpha\in[k]$.  Since $a_\alpha\parallel_P b_\alpha$, 
    $a_\alpha\le_P b_{\alpha+1}$, and 
    $a_{\alpha}$ is in the interior of $\cgB_{\alpha}$,
    we assert that $b_{\alpha+1}$ is also in 
    the interior of $\cgB_\alpha$.  If this assertion fails, then a witnessing path 
    from $a_\alpha$ to $b_{\alpha+1}$ would contain a point $z$ from the boundary 
    of $\cgB_\alpha$.  This would force $a_\alpha< z\le b_\alpha$ in $P$.  Clearly, this
    statement is false.  Thus, indeed $b_{\alpha+1}$ is in 
    the interior of $\cgB_\alpha$.

    Using Proposition~\ref{pro:shadow-comp}, this implies the following inclusion relations: 
    \[ \shad_j(b_{\alpha+1})\subseteq\shad_j(\max(\cgB_\alpha))\subseteq\shad_j(b_\alpha).  
    \] 
    Since these inclusions hold for all $\alpha\in[k]$, cyclically, we conclude that 
    for all $\alpha\in[k]$, 
    \[ 
    \shad_j(b_{\alpha+1})=\shad_j(\max(\cgB_\alpha))=\shad_j(b_\alpha).  
    \] 

    Since $\shad_j(b_{\alpha})$ are the same for all $\alpha\in[k]$, 
    their terminal blocks also coincide. 
    Let $\cgB$ be the common terminal block.
    The fact that $\shad_j(\max(\cgB_\alpha))= 
    \shad_j(b_\alpha)$ implies that $\cgB_\alpha$ is the terminal block of 
    $\shad_j(b_\alpha)$, so $\cgB=\cgB_{\alpha}$ for all $\alpha\in[k]$. 
    The fact that $\shad_j(b_{\alpha+1})=\shad_j(\max(\cgB))$ 
    coupled with the fact that $b_{\alpha+1}$ 
    is in the interior of $\cgB$,  
    together with Proposition~\ref{pro:shadow-comp} imply 
    $b_{\alpha+1}> \max(\cgB)$ in $P$.
    With these observations, the
    proof of the claim is complete.
  \end{proof}

  Let $y=\max(\cgB)$, and let $M$ be the set of all $\alpha\in[k]$ such that 
  $j_\alpha=j$. Since $j=\min\{j_\alpha:\alpha\in[k]\}$, we know that $M\neq\emptyset$.
  \begin{claim}\label{clm:cgI:2}
    If $\alpha\in M$, then $a_\alpha\parallel_P y$ and $\sd(a_\alpha)=j$.
  \end{claim}
  \begin{proof} 
    Let $\alpha\in M$.   Since the depth of $(a_{\alpha},b_{\alpha})$ is 
    $j$, we have $\shad_j(a_\alpha)\subsetneq \shad_j(b_\alpha)=\shad_j(y)$.  
    If $a_\alpha\leq_P y$, then $a_\alpha\leq_P b_\alpha$, which is false.  
    If $a_\alpha>_P y$, then Proposition~\ref{pro:shadow-comp} implies $\shad_j(a_\alpha)=\shad_j(b_\alpha)$, 
    which is false.  We conclude that $a_\alpha\parallel_P y$.

    Now suppose that $\sd(a_\alpha)>j$.  Let $\cgB'$ be the terminal block of 
    $\shad_j(a_\alpha)$.  Since $a_\alpha$ is in the interior of $\cgB$, and 
    $a_\alpha\parallel_P y$, it follows that $\cgB'\subsetneq \cgB$.  In particular, 
    since $w<_P a_\alpha$ for all elements $w$ on the boundary of $\cgB'$, it follows 
    that $y$ is in the exterior of $\cgB'$.  
    Furthermore, $\sd(a_{\alpha})>j$ implies that 
    $a_\alpha$ is in the 
    interior of $\cgB'$.  Since $a_\alpha\le_P b_{\alpha+1}$, it follows 
    that $b_{\alpha+1}$ is also in the interior of $\cgB'$.  

    Now let $W$ be a witnessing path from $y$ to $b_{\alpha+1}$.  Then 
    $W$ contains a point $w$ of $P$ that is on the boundary of $\cgB'$.  
    This implies $y< w< a_\alpha$ in $P$ so in particular $y< a_\alpha$ in $P$, which 
    is false.  With this observation, the proof of the claim is complete.
  \end{proof}
  \begin{claim}\label{clm:cgI:3}
    $M=[k]$. 
  \end{claim}
  \begin{proof}
  If the claim fails, cyclically relabel so that $j_1=j$ and $j_2>j$.
    The common parity gives $j_2\ge j+2$; this is the only use of parity
    in this lemma.  Let $\cgB'$ be the terminal block of
    $\shad_{j+1}(a_2)=\shad_{j+1}(b_2)$.
    Both $a_2$ and $b_2$ lie in its interior, and
    \[
      y\le\min(\cgB')<\max(\cgB')<\{a_2,b_2\}\ \textrm{in $P$}.
    \]
    Since $a_1\parallel y$ in $P$ by Claim~\ref{clm:cgI:2},
    Proposition~\ref{pro:uWLWR} implies $a_1\notin\cgB'$.
    A witnessing path from $a_1$ to $b_2$ must therefore meet the
    boundary of $\cgB'$ at a vertex $z$, giving
    $a_1< z\le\max(\cgB')< a_2$ in $P$.
    This contradicts the antichain of first coordinates in a strict
    alternating cycle.
\end{proof}
  The four statements of the lemma now follow directly from the
  three claims, so the proof of Lemma~\ref{lem:cgI-comprehensive} is complete. 
\end{proof}

\subsection{Separating Paths in Shadow Blocks}
When $\cgB$ is a shadow block, we let $x_\cgB=\min(\cgB)$, $y_\cgB=\max(\cgB)$.
We have already noted that if $u\in P$ and $u\in\cgB$, then 
Proposition~\ref{pro:uWLWR} implies that $x_\cgB$ belongs
to $\Wleft(u)$ and $\Wright(u)$. 
In particular, this implies $x_\cgB\le_P u$.
Now let $u$ be an element of $P$ that is in $\cgB$.
We assign $u$ to $A_\cgB$ if $u\parallel_P y_\cgB$; we assign $u$ to $B_\cgB$ 
if $u>_P y_\cgB$; and we assign $u$ to $Z_\cgB$ if $u\leq_P y_\cgB$.
Evidently, the three sets $A_\cgB$, $B_\cgB$, and $Z_\cgB$ are 
pairwise disjoint and they partition the 
set of elements of $P$ being in $\cgB$.
We record a useful observation: $A_\cgB\cup B_\cgB$ is an upset of $P$.
Indeed, if $u$ belongs to this union, then $u\not\le_P y_\cgB$.
A witnessing path starting at $u$ cannot reach the boundary of $\cgB$,
whose vertices are all at most $y_\cgB$.  Its endpoint remains in the
interior of $\cgB$ and is not at most $y_\cgB$.

\begin{proposition}\label{pro:all-inc-are-inside}
Let $\cgB$ be a shadow block and 
let $j=\sd(y_{\cgB})$. 
Let $a\in A_{\cgB}$, $b\in B_{\cgB}$, and $a\parallel b$ in $P$. 
Then $(a,b)$ is an inside pair with address $(j,\cgB)$.
\end{proposition}

\begin{proof}
  The common lower-shadow observation stated after the proof of
  Corollary~\ref{cor:b-and-c-have-the-same-shadow} gives
  $\shad_i(a)=\shad_i(b)$ for $0\le i<j$. 
  Proposition~\ref{pro:shadow-comp} gives
  \[
    \shad_j(a)\subsetneq\shad_j(y_\cgB)=\shad_j(b),
  \]
  since $a\parallel y_\cgB$ and $b> y_\cgB$ in $P$.
  Thus the depth of $(a,b)$ is $j$.
  The same proposition makes $\cgB$ the terminal block of $\shad_j(b)$.
  Since $a$ is in the interior of $\cgB$ Proposition~\ref{pro:geometry-inside} 
  now gives the required inside pair and address.
\end{proof}

Let $\cgB$ be a shadow block, and let
$N$ be a path in $G$ (not necessarily a witnessing path) from $x_\cgB$ to $y_\cgB$ such that all edges
on $N$ belong to $\cgB$.
When $u$ is either a vertex or edge of $P$ that
is in $\cgB$ (always including the boundary),  we can classify $u$ uniquely as being (1)~on the path $N$;
(2)~left of $N$; or (3)~right of $N$, using the following scheme.
The meaning of the first of these three options is clear.  Now
suppose that $u$ is not on $N$. We will say that $u$ is \emph{left of $N$} 
if $u$ is in a region in the plane bounded by a cycle formed by two paths,
with one path a portion of 
$N$ and the other path a portion of the left side of $\cgB$.  Symmetrically, we
say that $u$ is \emph{right of $N$} if $u$ is in a region in the plane bounded 
by two paths, with one path a portion of $N$ and the other path a portion of the 
right side of $\cgB$.  We illustrate these conventions on the
left side of Figure~\ref{fig:N-left-right}.

\begin{figure}
  \centering
  \includegraphics[scale=.8]{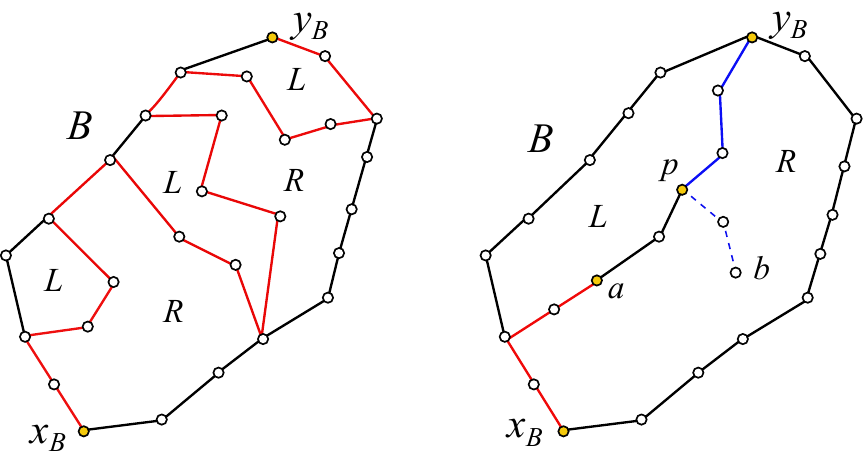}
  \caption{Left:  The red path $N$ from $x_\cgB$ to $y_\cgB$ splits
  the shadow block $\cgB$ into regions.  The regions marked $L$ are left of $N$,
  and the regions marked $R$ are right of $N$.  Note that each region is a cycle
  whose boundary consists of two paths, one a subpath of $N$, and the other a
  subpath of one of the two sides of the shadow block $\cgB$. 
  Right:  We show a separating path $N$ associated with the comparability $a<_P b$.
  The element $p$ is the peak of $N$, and $p\le_P b$.  When $b\neq p$, the element
  $b$ can be on either side of $N$.
  }
  \label{fig:N-left-right}
\end{figure}

Now let $\cgB$ be a shadow block, and let
$(a,b)\in A_\cgB\times B_\cgB$ with $a<_P b$. 
A path $N$ from $x_{\cgB}$ to $y_{\cgB}$ in $G$ 
is a \emph{separating path associated with $a<b$ in $P$} if
the following statements hold:
\begin{enumerate}
  \item $a$ is on $N$ and $x_\cgB Na$ is the suffix of $\Wleft(a)$ starting at $x_\cgB$.
  \item  There is an element $p$ of $B_\cgB$ that is on $N$ 
    such that $pNy_\cgB$ is the part of $\Wleft(b)$ from $y_{\cgB}$ to $p$ traversed backwards.
  \item $a<_P p$, and $aNp$ is a witnessing path from $a$ to $p$.
\end{enumerate}
The element $p$ referenced in this definition 
is called the \emph{peak} of $N$.
Note that $z\le p\le b$ in $P$, 
for every element $z$ of $P$ that is on the path $N$.
Note that all edges of $N$ are in $\cgB$ (by Proposition~\ref{pro:path-in-block}). 

Note also that when $\cgB$ is a shadow block and $(a,b)\in A_{\cgB}\times B_{\cgB}$ with $a<b$ in $P$, then there is always 
a separating path associated with $a< b$ in $P$. 
Indeed, Proposition~\ref{pro:shadow-block-maxB} puts $y_\cgB$ on $\Wleft(b)$.  
Choose a witnessing path from $a$ to $b$ and let $p$ be
its first vertex on $y_\cgB\Wleft(b)b$.
Concatenate $x_\cgB\Wleft(a)a$, the chosen path from $a$ to $p$, and
$p\Wleft(b)y_\cgB$ traversed backwards.  The first two parts meet
only at $a$ by acyclicity; the last part meets the middle part only at
$p$ by its choice, and cannot meet the first part because
$a\parallel_P y_\cgB$.  Thus the concatenation is a path and has all
the required properties. 
The concept of a separating path associated
with an inequality $a<_P b$ is illustrated on the right side of Figure~\ref{fig:N-left-right}.

\begin{proposition}\label{pro:force-comp}
  Let $\cgB$ be a shadow block, let 
  $(a,b)\in A_\cgB\times B_\cgB$ with $a<_P b$, and let $N$ be a separating path 
  associated with $a<_P b$.  If $u$ and $v$ are
  elements of $P$ that belong to $\cgB$, $u<_P v$, $u$ is on one side of $N$,
  and $v$ is on the other side, then $u<_P b$.
\end{proposition}
\begin{proof}
  Let $W$ be a witnessing path from $u$ to $v$. 
  Proposition~\ref{pro:path-in-block} implies that 
  all vertices and edges of $W$ are in $\cgB$.
  Thus there must be
  an element  $z$ of $P$ common to $W$ and $N$.  If $p$ is the peak of $N$,
  this implies $u< z\le p\le b$ in $P$.  Therefore, $u< b$ in $P$.
\end{proof}
\begin{proposition}\label{pro:force-sides}
  Let $\cgB$ be a shadow block, let 
  $(a,b)\in A_\cgB\times B_\cgB$ with $a<_P b$, and let $N$ be a separating path 
  associated with $a<_P b$.  If  $(a',b')\in A_\cgB\times B_\cgB$, then the
  following statements hold:
  \begin{enumerate}
    \item If $(a',a)$ is a left pair, then $a'$ is left of $N$.
    \label{pro:item:a'-left}
    \item If $(a',a)$ is a right pair, then $a'$ is right of $N$.
    \label{pro:item:a'-right}
    \item If $(b',b)$ is a left pair, and $b'\parallel_P a$, then $b'$ is right of $N$.
    \label{pro:item:b'-left}
    \item If $(b',b)$ is a right pair, and $b'\parallel_P a$, then $b'$ is left of $N$.
    \label{pro:item:b'-right}
\end{enumerate}
\end{proposition}

\begin{proof} 
  The arguments for statements~\ref{pro:item:a'-left} and~\ref{pro:item:a'-right} are symmetric, so we only prove statement~\ref{pro:item:a'-left}.
  Let $w$ be the largest point of $P$ common to 
  $x_{\cgB}\Wleft(a')a'$ and $x_{\cgB}\Wleft(a)a$. 
  Since $(a',a)$ is a left pair, it is in particular an incomparable pair in $P$ (see Proposition~\ref{pro:details-on-LR}).
  Thus, $w<\set{a,a'}$ in $P$.
  Let $e_0$ be the edge common to $\Wleft(a')$ and $\Wleft(a)$ that is immediately before $w$ 
  (and in the case $w=x_{0}$ let $e_0=e_{-\infty}$).
  Let $e'$ and $e$ be the edges immediately after $w$ on $\Wleft(a')$ and
  $\Wleft(a)$, respectively. 
  Since $(a',a)$ is a left pair, we conclude that 
  $e'$ is left of $e$ in the $(w,e_0)$-ordering. 
  Note that $e_0$ is either on $N$ (if $w\neq x_{\cgB}$) 
  or $e_0$ is outside $\cgB$ (if $w=x_{\cgB}$). 
  In either case, since $e$ is on $N$, 
  we conclude that $e'$ is left of $N$.

  Now we assume that $a'$ is not left of $N$ and
  argue to a contradiction. 
  This requires that there is a vertex $v$ with
  $v\neq w$ common to $w\Wleft(a')a'$ and $N$. 
  The choice of $w$ implies
  $v\not\in wNa = w\Wleft(a)a$. 
  Now suppose that $v\in aNp$ where $p$ is the peak of $N$.  Then
  $a < v\le_P a'$ which is a contradiction as $a\parallel a'$ in $P$. 
  The last possibility is that $v\in pNy_{\cgB}$. 
  Now we have $y_{\cgB}\le v\le a'$ in $P$, so $y_{\cgB}\leq a'$ in $P$,
  which contradicts the fact that $a'\in A_{\cgB}$. 
  The contradiction completes the proof of~\ref{pro:item:a'-left}.

  The proofs of statements~\ref{pro:item:b'-left} and~\ref{pro:item:b'-right} are symmetric, so we only prove statement~\ref{pro:item:b'-left}.  
  Let $w$ be the largest element of $P$ common to 
  $y_\cgB \Wleft(b)b$ and $y_\cgB \Wleft(b')b'$.  
  Since
  $b'\parallel_P b$, we must have $w<_P \set{b,b'}$. 
  Also $w< p$ in $P$ as otherwise $a< p\leq w \leq b'$ in $P$, which
  is false. 
  Now we conclude (like in the paragraph before) 
  that the edge of $\Wleft(b')$ that is immediately
  after $w$ is right of $N$.  
  Next we assume that $b'$ is not right of $N$ and
  argue to a contradiction.  Now we must have an element $v$ of $P$ with
  $v\neq w$ such that $v$ is common to $w\Wleft(b')b'$ and $N$.  The definition
  of $w$ implies that $v\not\in pNy_{\cgB}$.  
  If $v\in aNp$, then $a\le v\le b'$ in $P$.
  This implies $a<_P b'$, which is false.  It follows that
  $v\in x_\cgB Na$.  Now we have $y_\cgB< v\leq a$ in $P$.  
  This implies $y_\cgB<_P a$, which is
  false. The contradiction completes the proof of~\ref{pro:item:b'-left}. 
\end{proof}

\begin{proposition}\label{pro:block-details}
  Let $\cgB$ be a shadow block, and let $j=\sd(y_\cgB)$.  If
  $a,a'\in A_\cgB$,  and $b,b'\in B_\cgB$, then the following statements hold:
  \begin{enumerate}
    \item $b\not\in\shad_j(a)$.
    \label{pro:item:b-not-in-shad-a}
    \item $a\not\in\shad_{j+1}(b)$.
    \label{pro:item:a-not-in-shad-b}
    \item If $a\parallel_P a'$, and $a\in\shad_j(a')$, then $a\parallel_P b$. \label{pro:item:incomp}
    \label{pro:item:a-in-shad-a'}
    \item If $\sd(a)>j$ then $a\parallel_P b$.
    \label{pro:item:sda-more-than-j}
    \item If $b\in\shad_{j+1}(b')$, and $a<_P b$, then
      $a<_P b'$.
    \label{pro:item:b-in-shad-b'}
    \item If $\shad_{j+1}(b)=\shad_{j+1}(b')$, 
      then $a<_P b$ if and only if $a<_P b'$.
    \label{pro:item:b-and-b'-same-shadow}
  \end{enumerate}
\end{proposition}
\begin{proof}
  First note that $b< a$ in $P$ is impossible for $a\in A_\cgB$ and
  $b\in B_\cgB$, since it would imply $y_\cgB< b< a$ in $P$.
  Thus, in statements~\ref{pro:item:a-in-shad-a'} and
  \ref{pro:item:sda-more-than-j}, it suffices to exclude $a< b$ in $P$.
  For the proof of the statement~\ref{pro:item:b-not-in-shad-a}, 
  suppose to the contrary that $b\in\shad_j(a)$.
  Since $a\parallel_P y_\cgB$, 
  Proposition~\ref{pro:shadow-comp} implies that $\shad_{j}(a)\subsetneq \shad_j(y_{\cgB})$ 
  and in particular $y_\cgB\not\in \shad_j(a)$. 
  Let $W$ be a 
  witnessing path from $y_\cgB$ to $b$. 
  Since $y_{\cgB}\not\in\shad_j(a)$ and $b\in\shad_j(a)$,
  the path $W$ must contain an element $z$ of $P$ such that
  $z$ is also on the boundary of $\shad_j(a)$. 
  This implies $y_{\cgB}< z\le a$ in $P$, and
  clearly this statement contradicts $a\in A_{\cgB}$. 
  The contradiction completes the proof of statement~\ref{pro:item:b-not-in-shad-a}.

  For the proof of statement~\ref{pro:item:a-not-in-shad-b}, 
  suppose to the contrary that $a\in\shad_{j+1}(b)$.  
  Note that $x_{\cgB}\not\in\shad_{j+1}(b)$ and $x_\cgB<_P a$.  
  This implies that a witnessing
  path $W$ from $x_\cgB$ to $a$ contains 
  an element $z$ of $P$ that is also on the boundary of
  $\shad_{j+1}(b)$.  This implies $y_{\cgB}\le z \le a$ in $P$, 
  which contradicts that $a\in A_{\cgB}$.
  The contradiction completes the proof of 
  statement~\ref{pro:item:a-not-in-shad-b}.

  For the proof of statement~\ref{pro:item:a-in-shad-a'}, suppose that $a\parallel_P a'$,
  $a\in\shad_j(a')$, and $a<_P b$. 
  By~\ref{pro:item:b-not-in-shad-a} we have $b\not\in\shad_j(a')$ and 
  therefore
  a witnessing path from $a$ to
  $b$ must contain a point $z$ from the boundary of $\shad_j(a')$.
  This implies $a< z\le a'$ in $P$, which is false.  
  The contradiction completes the proof of statement~\ref{pro:item:a-in-shad-a'}.

  For the proof of statement~\ref{pro:item:sda-more-than-j}, 
  suppose that $\sd(a)>j$ and $a<b$ in $P$.
  Therefore 
  $a$ is in the interior of $\shad_j(a)$ and by statement~\ref{pro:item:b-not-in-shad-a}, we have  
  $b$ is not in $\shad_j(a)$. 
  Let $W$ be a witnessing path from $a$ to $b$. 
  There must be an element $z$ of $W$ that is on the boundary of $\shad_j(a)$. This is a clear contradiction as all elements on the boundary of $\shad_j(a)$ are below $a$ in $P$.

  For the proof of statement~\ref{pro:item:b-in-shad-b'}, suppose that 
  $b\in\shad_{j+1}(b')$, and $a<_P b$.  
  By~\ref{pro:item:a-not-in-shad-b} we have that 
  $a\not\in\shad_{j+1}(b')$.
  A witnessing path from $a$ to $b$ 
  must contain an element $z$ of $P$ that
  is on the boundary of $\shad_{j+1}(b')$.
  This implies $a< z\le b'$ in $P$. 
  This completes the proof of statement~\ref{pro:item:b-in-shad-b'}.

  Statement~\ref{pro:item:b-and-b'-same-shadow} follows immediately 
  from~\ref{pro:item:b-in-shad-b'} as 
  $\shad_{j+1}(b)=\shad_{j+1}(b')$ implies $b\in\shad_{j+1}(b')$ and 
  $b'\in\shad_{j+1}(b)$.

\end{proof}

\begin{proposition}\label{pro:sac-comprehensive}
  Let $\theta\in\{0,1\}$ and let $((a_1,b_1),\dots,(a_k,b_k))$ be a 
  strict alternating cycle in $P$ of pairs from $I_\theta$.  Let $(j,\cgB)$ be
  the common address of the pairs on the cycle.
  If $\alpha$ and $\beta$ are distinct integers in $[k]$,
  then the following statements hold:
  \begin{enumerate}
    \item $a_\alpha\not\in\shad_j(a_\beta)$.
    \label{pro:item-a-not-in-shadow-a}
    \item $b_\alpha\not\in\shad_{j+1}(b_\beta)$.
    \label{pro:item-b-not-in-shadow-b}
    \item $(a_\alpha,a_\beta)$ is either a left pair or a right pair.
    \label{pro:item-a-left-or-right}
    \item $(b_\alpha,b_\beta)$ is either a left pair or a right pair.
    \label{pro:item-b-left-or-right}
    \item $(a_\alpha,a_\beta)$ is a left pair if and only if 
      $(b_{\alpha+1},b_{\beta+1})$ is a right pair.
    \label{pro:item-a-left-iff-b-right}     
  \end{enumerate}
\end{proposition}
\begin{proof}
  We note that Lemma~\ref{lem:cgI-comprehensive} implies
  $a_\alpha\in A_\cgB$ and $b_\alpha\in B_\cgB$, for all
  $\alpha\in[k]$.  Also the strictness of the alternating cycle implies
   that $\{a_1,\dots,a_k\}$ and $\{b_1,\dots,b_k\}$ 
   are $k$-element antichains in $P$.

  For the proof of statement~\ref{pro:item-a-not-in-shadow-a}, 
  since $a_\alpha\parallel a_\beta$,
  and $a_\alpha< b_{\alpha+1}$ in $P$, 
  Proposition~\ref{pro:block-details}.\ref{pro:item:a-in-shad-a'} 
  implies $a_\alpha\not\in\shad_j(a_\beta)$.

  For the proof of statement~\ref{pro:item-b-not-in-shadow-b}, 
  since $a_{\alpha-1}< b_\alpha$, 
  and $a_{\alpha-1}\parallel b_\beta$ in $P$,
  Proposition~\ref{pro:block-details}.\ref{pro:item:b-in-shad-b'} 
  implies $b_\alpha\not\in\shad_{j+1}(b_\beta)$.

  For the proof of statement~\ref{pro:item-a-left-or-right}, 
  recall that $(a_{\alpha},a_{\beta})$ is either left, or right, 
  or inside, or outside. 
  Recall that the $i$-shadows of $a_{\alpha}$ and $a_{\beta}$ are the same 
  for all $0\leq i<j$, and 
  by statement~\ref{pro:item-a-not-in-shadow-a} we have 
  $a_{\alpha}\not\in\shad_j(a_\beta)$ and 
  $a_{\beta}\not\in\shad_j(a_\alpha)$. 
  Now Proposition~\ref{pro:geometry-inside} 
  implies that $(a_{\alpha},a_{\beta})$ is not an inside pair and 
  $(a_{\beta},a_{\alpha})$ is not an inside pair. 
  This forces $(a_{\alpha},a_{\beta})$ to be a left pair or a right pair, 
  as desired. 
  Statement~\ref{pro:item-b-left-or-right} follows along the same lines.

  For the proof of statement~\ref{pro:item-a-left-iff-b-right}, 
  we may assume, without loss of generality that
  $(a_\alpha,a_\beta)$ is a left pair.
  Let $N$ be a separating path in the block $\cgB$ associated with 
  $a_\beta< b_{\beta+1}$ in $P$. 
  Then Proposition~\ref{pro:force-sides} implies 
  that $a_\alpha$ is left of $N$.

  We already know that 
  $(b_{\alpha+1},b_{\beta+1})$ is either a left pair or
  a right pair.  Suppose that it is a left pair.  
  Since $b_{\alpha+1}\parallel a_\beta$ in $P$, 
  Proposition~\ref{pro:force-sides} implies that $b_{\alpha+1}$ is
  right of $N$. 
  Now we have $a_\alpha$ left of $N$, $b_{\alpha+1}$ right
  of $N$, and $a_\alpha< b_{\alpha+1}$ in $P$. 
  Proposition~\ref{pro:force-comp}
  now implies $a_\alpha< b_{\beta+1}$ in $P$, 
  but this is false.  We conclude
  that $(b_{\alpha+1},b_{\beta+1})$ is a right pair, as desired.
  With this observation, the proof of statement~\ref{pro:item-a-left-iff-b-right} is complete.
\end{proof}

The next statement is just a special case of the preceding proposition, but 
it deserves to be highlighted.
\begin{corollary}\label{cor:sac-2}
  Let $\theta\in\{0,1\}$ and let $((a_1,b_1),(a_2,b_2))$ be a 
  strict alternating cycle in $P$ of inside pairs from $I_\theta$. 
  Then  $(a_1,a_2)$ is a left pair if and only if $(b_1,b_2)$ is a left pair.
\end{corollary}

With these observations, we have reached the end of the common part of the
proofs of our two main theorems.
\section{Boolean Dimension is Bounded}\label{sec:bdim}
In this section, we give the second part of the proof showing that
if $P$ is a poset with a unique
minimal element and a planar cover graph, then $\bdim(P)\le 10$.
The initial setup is the same as in the last section.  We assume
that $P$ is a poset with a planar cover graph and a unique minimal element denoted
$x_0$.  We fix a plane drawing of the cover graph $G$ of $P$ with
$x_0$ on the exterior face.  Ultimately, we will show that
$\bdim(P)\le 10$ by constructing a Boolean 
realizer $(L_1,\dots,L_{10})$ with an appropriate $10$-ary Boolean formula.

Let $\theta\in\{0,1\}$.  Recall that $I_\theta$ denotes the set of all
inside pairs $(a,b)$ such that if $(j,\cgB)$ is the address of
$(a,b)$, then $j\equiv\theta\mod2$.   We then define the following two
sets:
\begin{enumerate}
  \item $X_\theta(\textrm{inside left-safe})$ consists of all pairs $(a,b)\in I_\theta$ for
    which there does not exist a pair $(a',b')\in I_\theta$ such that
    $a<_P b'$, $a'<_P b$, and $(a',a)$ is a left pair.
  \item $X_\theta(\textrm{inside right-safe})$ consists of all pairs $(a,b)\in I_\theta$ for
    which there does not exist a pair $(a',b')\in I_\theta$ such that
    $a<_P b'$, $a'<_P b$, and $(a',a)$ is a right pair.
\end{enumerate}

\begin{proposition}\label{pro:LR-safe}
  For each $\theta\in\{0,1\}$, the sets
\[
X_\theta(\textrm{inside left-safe})\quad \textrm{and}\quad X_\theta(\textrm{inside right-safe})
\]
are reversible in $P$.
\end{proposition}
\begin{proof}
  Let $\theta\in\{0,1\}$.  We show that $X_\theta(\textrm{inside left-safe})$
  is reversible in $P$.  The argument for $X_\theta(\textrm{inside right-safe})$
  is symmetric.  Suppose to the contrary that 
  $((a_1,b_1),\dots,(a_k,b_k))$ is a strict alternating cycle in $P$ of pairs from 
  $X_\theta(\textrm{inside left-safe})$. 
  Using Lemma~\ref{lem:cgI-comprehensive}, we know that there is a pair
  $(j,\cgB)$ that is the address of all pairs on the cycle.  Also, we know
  that for all $\alpha\in[k]$, $a_\alpha\in A_\cgB$, and $b_{\alpha}\in B_\cgB$.  
  By Proposition~\ref{pro:all-inc-are-inside} it follows that 
  if $\alpha,\beta\in[k]$, and $\beta\neq\alpha+1$, then
  $(a_\alpha,b_\beta)$ is an inside pair whose address is $(j,\cgB)$. 
  In particular, all such pairs are in $I_\theta$.

  Proposition~\ref{pro:sac-comprehensive}  implies that for distinct
  $\alpha,\beta\in[k]$, $(a_\alpha,a_\beta)$ is either a left pair or a right pair.
  Since left (right) pairs are transitive (see Proposition~\ref{pro:details-on-LR}), 
  we may assume, after a relabeling if necessary, that $(a_\alpha,a_1)$ is a left pair, 
  for every $\alpha\in\set{2,\dots,k}$.  
  Now we know that $(a_k,b_2)$ is in $I_\theta$,
  $a_1< b_2$ in $P$,  $a_k< b_1$ in $P$, and $(a_k,a_1)$ is a left pair.  This contradicts
  the assumption that $(a_1,b_1)$ is a pair in $X_\theta(\textrm{inside left-safe})$.
  The contradiction completes the proof.
\end{proof}

Propositions~\ref{pro:LO-RO} and~\ref{pro:LR-safe} identify the 
following six sets as being reversible in $P$:
\begin{align*}
X_1&=X(\textrm{left})\cup X(\textrm{outside}),\\
X_2&=X(\textrm{right})\cup X(\textrm{outside}),\\  
X_3&=X_0(\textrm{inside left-safe}),\\
X_4&=X_0(\textrm{inside right-safe}),\\
X_5&=X_1(\textrm{inside left-safe}),\\
X_6&=X_1(\textrm{inside right-safe}).
\end{align*}
For brevity, we say that a pair $(a,b)\in\Inc(P)$ is \emph{dangerous}
when it does not belong to $X_1\cup\dots\cup X_6$.  We state for emphasis the 
following elementary characterization of dangerous pairs. 
(Recall Corollary~\ref{cor:sac-2} and Lemma~\ref{lem:cgI-comprehensive}.)

\begin{proposition}\label{pro:dangerous}
  Let $(a,b)\in \Inc(P)$. 
  Then $(a,b)$ is dangerous
  if and only if 
  \begin{enumerate}
    \item $(a,b)$ is an inside pair; and 
    \item \label{item:ii}
    if $(a,b)$ has address $(j,\cgB)$, then $a\in A_{\cgB}$, $b\in B_{\cgB}$, and 
    there are inside pairs $(w,z)$ and $(w',z')$ also with address $(j,\cgB)$ such that 
    \begin{enumerateAlpha}
    \item $a< z$, $w< b$ in $P$, $(w,a)$ is a left pair, and $(z,b)$ is a left pair;
    \item $a< z'$, $w'< b$ in $P$, and $(w',a)$ is a right pair, and $(z',b)$ is also a right pair.
    \end{enumerateAlpha}
    Moreover, $\cgB$ is the terminal block of $\shad_j(b)=\shad_j(z)=\shad_j(z')$.
  \end{enumerate}
\end{proposition}

When $(a,b)$ is a dangerous pair with address $(j,\cgB)$, as
evidenced by the pairs $(w,z)$ and $(w',z')$ from
Proposition~\ref{pro:dangerous}, we will refer
to $(w,z)$ as a \emph{left neighbor} of $(a,b)$. 
 Analogously, $(w',z')$ will
be called a \emph{right neighbor} of $(a,b)$.  
Note that by Lemma~\ref{lem:cgI-comprehensive},
 we have $\sd(w)=\sd(w')=\sd(a)=j$.

For $i\in[6]$, 
we let $L_i$ be a linear extension of $P$ that
reverses the incomparable pairs in $X_i$ 
(the linear extensions $L_1$ and
$L_2$ were discussed in the common part of the proof given in the preceding section).
Given a pair $(a,b)$ of elements of $P$, 
we can say with
certainty that $a\not\le b$ in $P$ 
if there is some $i\in[6]$ such that $a\not\le b$ in $L_i$.
However, if $a\le b$ in $L_i$ for all $i\in[6]$, we know that one of the
following two statements holds:
\begin{enumerate}
  \item $a\le b$ in $P$; or
  \item $(a,b)$ is a dangerous pair.
\end{enumerate}
The goal for the next four linear orders $L_7,\dots,L_{10}$ will be to detect which of these
two options applies. These orders need not be linear extensions of $P$.
We continue exploring the properties of dangerous pairs in $P$.

\begin{proposition}\label{pro:j+1-dangerous-equiv}
  Let $\cgB$ be a shadow block, and let $j=\sd(y_\cgB)$.  If $a\in A_\cgB$,
  $b,c\in B_\cgB$, and $\shad_{j+1}(b)=\shad_{j+1}(c)$, then 
  $(a,b)$ is dangerous if and only if $(a,c)$ is dangerous.
\end{proposition}
\begin{proof}
  Recall, that Proposition~\ref{pro:all-inc-are-inside} states that 
  all incomparable pairs from $A_\cgB\times B_\cgB$ are inside pairs with
  address $(j,\cgB)$.  
  Let $a\in A_\cgB$,
  $b,c\in B_\cgB$, $\shad_{j+1}(b)=\shad_{j+1}(c)$, and assume that  
  $(a,b)$ is dangerous.
  Let $(w,z)$ and $(w',z')$ be, respectively,
  a left neighbor and a right neighbor of $(a,b)$.  Then
  $a< \{z,z'\}$ in $P$, 
  $b>\{w,w'\}$ in $P$, 
  $(w,a)$ is a left pair, 
  and $(w',a)$ is a right pair. 
  Since $\shad_{j+1}(b)=\shad_{j+1}(c)$, 
  Proposition~\ref{pro:block-details}.\ref{pro:item:b-and-b'-same-shadow} implies that if $u\in A_\cgB$, then
  $u< b$ in $P$ if and only if $u< c$ in $P$. 
  In particular, $a\parallel b$ in $P$ implies $a\parallel c$ in $P$, so
  $(a,c)$ is an inside pair with address $(j,\cgB)$. 
  Also, $b>\{w,w'\}$ in $P$ implies 
  $c>\{w,w'\}$ in $P$. 
  Corollary~\ref{cor:sac-2}, applied to the resulting strict alternating
  $2$-cycles, supplies the required left and right relations of $z,z'$ to $c$. 
  Therefore, the pairs $(w,z)$ and $(w',z')$ now witness that
  $(a,c)$ is a dangerous pair. 
  Since this argument can be reversed if we start
  with the assumption that $(a,c)$ is dangerous, the proof is complete.
\end{proof}

For each $\theta\in\{0,1\}$, let
$D_\theta$ denote the set of dangerous pairs $(a,b)$ such that 
$(a,b)\in I_{\theta}$. 
Note that if $(a,b)$ is a dangerous pair with address $(j,\cgB)$, 
then Lemma~\ref{lem:cgI-comprehensive} implies that $\sd(a) = j$. 
Therefore, $D_{\theta}$ is exactly the set of 
all dangerous inside pairs $(a,b)$ with 
$\sd(a)\equiv\theta\mod2$.
The next proposition implies that when treated as a binary relation, the
set $D_\theta$ is transitive.

\begin{proposition}\label{pro:dangerous-transitive}
  For each $\theta\in\{0,1\}$, 
  if $(a,b),(b,c)\in D_\theta$, 
  then $(a,c)\in D_\theta$.
\end{proposition}
\begin{proof}
Fix $\theta\in\{0,1\}$ and suppose $(a,b),(b,c)\in D_\theta$.
  Let $(j,\cgB)$ be the address of $(a,b)$, so $j=\sd(a)$.
  Since $\sd(b)>j$ and $\sd(b)\equiv j\pmod2$, we have $\sd(b)\ge j+2$.
  The depth of the dangerous pair $(b,c)$ is $\sd(b)$, so, directly
  from the definition of depth,
  $\shad_{j+1}(b)=\shad_{j+1}(c)$ and $\shad_j(b)=\shad_j(c)$.
  In particular, $c\in B_\cgB$: the common $j$-shadow has terminal
  block $\cgB$, and both elements have depth greater than $j$.
  Proposition~\ref{pro:j+1-dangerous-equiv} now shows that $(a,c)$ is
  dangerous with address $(j,\cgB)$, hence belongs to $D_\theta$.
\end{proof}

Let $a$ and $b$ be elements of $P$.
We will then say that the 
comparability $a< b$ in $P$ \emph{tilts right} when there exists an element
$u$ of $P$ satisfying the following conditions: (1)~$(u,a)$ is a right pair;
(2)~$(u,b)$ is a dangerous pair; and (3)~if $(j,\cgB)$ is the address of $(u,b)$, 
then $a\in A_\cgB$. 
Recall that by Proposition~\ref{pro:shadow-comp}.\ref{pro:item:depth} we have $\sd(a)\geq j$, and
by Proposition~\ref{pro:dangerous}.\ref{item:ii} we have $b\in B_\cgB$. 
This coupled with the inequality $a< b$ in $P$ 
implies (by Proposition~\ref{pro:block-details}.\ref{pro:item:sda-more-than-j}) 
that $\sd(a)=j$.

Analogously, we will then say that the comparability $a< b$ in $P$ \emph{tilts left} when 
there exists an element $v$ of $P$ satisfying the following conditions: (1)~$(v,a)$ 
is a left pair; (2)~$(v,b)$ is a dangerous pair; and (3)~if $(j,\cgB)$ is the address of $(v,b)$, 
then $a\in A_\cgB$.  Again, we have $\sd(a)=j$.
In either case $b\in B_\cgB$, so $\sd(b)>j=\sd(a)$.
Consequently, a comparability between elements of the same shadow depth
cannot tilt in either direction.

In the proof of the following proposition, the fact that a 
dangerous pair has \emph{both} a left neighbor and a right neighbor
is essential.

\begin{proposition}\label{pro:tilt-one-direction} 
  No comparability in $P$ tilts both left and right. 
\end{proposition}
\begin{proof}
Let $a$ and $b$ be elements of $P$ with $a< b$ in $P$. 
  We assume that the comparability $a< b$ in $P$ tilts both right and left and argue
  to a contradiction.  
  Let $u$ be an element of $P$ witnessing that
  $a< b$ in $P$ tilts right and 
  let $v$ be an element of $P$ witnessing that
  $a< b$ in $P$ tilts left. 
  Therefore, $(u,b)$ and $(v,b)$ are dangerous pairs. 

  Let $(j_u,\cgB)$ be the address of $(u,b)$ and 
  let $(j_v,\cgB')$ be the address of $(v,b)$. 
  We claim that $(j_u,\cgB)=(j_v,\cgB')$. 
  Indeed, since $v$ is a witness to $a<b$ tilting right, 
  we have that $a\in A_{\cgB'}$ and $b\in B_{\cgB'}$.
  Now since $a$ lies in $\cgB'$ we have $\sd(a)\geq j_v$. 
  Furthermore, $a<b$ in $P$ together with Proposition \ref{pro:block-details}.\ref{pro:item:sda-more-than-j} implies $\sd(a)=j_v$. 
  An analogous argument shows that $\sd(a)=j_u$. 
  Therefore, $j_u=j_v$.
  Now, by Proposition~\ref{pro:dangerous} both $\cgB$ and $\cgB'$ must be 
  the terminal block of $\shad_j(b)$.  
  Thus, $\cgB=\cgB'$, which completes the proof of the claim. 
  Let $j=j_u=j_v$. Thus both $(u,b)$ and $(v,b)$ have address $(j,\cgB)$. 
  In particular, $a,u,v\in A_\cgB$ and $b\in B_\cgB$.
  
  Let $N$ be a separating path in $\cgB$ that is associated with the 
  inequality $a< b$ in $P$, and let $p$ be the peak of $N$. 
  Since
  $(u,a)$ is a right pair, and $(v,a)$ is a left pair, 
  Proposition~\ref{pro:force-sides} implies that $u$ is right of $N$, and $v$ is left of $N$.
 
  Let $(w,z)$ be a left neighbor of $(v,b)$,
  and let $(w',z')$ be a right neighbor of $(u,b)$.  
  Recall that $w,w'\in A_{\cgB}$ and $z,z'\in B_{\cgB}$.

  Let $s$ be the largest element common to $\Wleft(p)$ and $\Wleft(z)$. 
  Clearly, $s$ lies in $y_{\cgB} Np$. 
  Then $s\neq z$ as otherwise $z=s\leq p \leq b$ in $P$. 
  Let $e^-$ be the edge of $\Wleft(z)$ before $s$. 
  Note that either $e^-$ is in $N$ or 
  $s=y_{\cgB}$ and then 
  $e^-$ is on the left side of $\cgB$. 
  Let $e^+$ be the edge of $\Wleft(z)$ immediately after
  $s$. Then $e^+$ is not on $N$, so we have identified two cases: 
  $e^+$ is right of $N$ and $e^+$ is left of $N$.

  \smallskip
  \noindent
  \textit{Case 1.}\quad $e^+$ is right of $N$.

  In this case, we are going to show that $z$ is also right of $N$. 
  Suppose the opposite, so $z$ is not right of $N$.  
  This implies that $s\Wleft(z)z$ intersects $N$, say at element $t$ with $t\neq s$. 
  First, if $s=p$ then $p\Wleft(z)tNp$ is a directed cycle in $P$, contradiction.
  Now, suppose that $s\neq p$ so $s < p$ in $P$. 
  Again, if $t\in y_\cgB Ns$ in $P$ then there is a directed cycle in $P$, contradiction. 
  Thus, $t\in sNx_\cgB$. 
  Let $e$ be the first edge of $sNp$. 
  Since $e^+$ is right of $N$, we have that $e^+$ is left of $e$ in the $(s,e^-)$-ordering. 
  Now the witnessing path $x_0\Wleft(p)s\Wleft(z)tNp$ is $x_0$-left of $\Wleft(p)$ which should be the leftmost path from $x_0$ to $p$, contradiction.

  Therefore, we have that $z$ is right of $N$. 
  Since $v$ is left of $N$, and $v< z$ in $P$, 
  Proposition~\ref{pro:force-comp} forces $v<b$ in $P$, which is false.  
  This shows that Case 1 cannot hold.

  \smallskip
  \noindent
  \textit{Case 2.}\quad $e^+$ is left of $N$.

  In this case, if we have $s<p$ in $P$, 
  then $e^+$ is right of $e$ in the $(s,e^-)$-ordering, 
  where $e$ is the first edge of $sNp$. 
  This would imply $\Wleft(z)$ being $x_0$-right of $\Wleft(b)$, 
  contradicting that $(z,b)$ is a left pair.

  Therefore, $s=p$ and all edges and vertices of $p\Wleft(z)z$, except $p$, 
  are left of $N$.
  Also, all edges and vertices of $p\Wleft(b)b$, except
  $p$, are left of $N$.

  Let $s'$ be the largest element common to $\Wleft(p)$ and $\Wleft(z')$.
  Clearly, $s'$ lies in $y_{\cgB} Np$. 
  Note that $s'\neq z'$ as otherwise $z'=s'\leq b$ in $P$ which is false. 
  Let $e'^-$ be the edge of $\Wleft(z')$ before $s'$. 
  Note that either $e'^-$ is in $N$ or 
  $s'=y_{\cgB}$ and then 
  $e'^-$ is on the left side of $\cgB$. 
  Let $e'^+$ be the edge of $\Wleft(z')$ immediately after
  $s'$. Then $e'^+$ is not on $N$.

  Let $e'$ be the first edge of $s'\Wleft(b)b$. 
  Now, the fact that $(z',b)$ is a right pair implies that 
  $\Wleft(z')$ is $x_0$-right of $\Wleft(b)$. 
  This implies that $e'^+$ is right of $e'$ in the $(s',e'^-)$-ordering.
  Since $e'$ is on $N$ (if $s'<p$ in $P$) or $e'$ is left of $N$ (if $s'=p$), 
  we conclude that $e'^+$ must be left of $N$.

  Consider now the case that $z'$ is left of $N$. 
  Recall that $u$ is right of $N$. 
  Since $u<z'$ in $P$, Proposition~\ref{pro:force-comp} forces $u<b$ in $P$, which is false.  
  This shows that $z'$ is not left of $N$.

  Since $e'^+$ is left of $N$ but $z'$ is not left of $N$, 
  there is an element $t'\neq s'$ in $s'\Wleft(z')z'$ 
  such that $t'$ is on $N$. 

The arc $s'\Wleft(z')t'$ has no internal vertex on $N$ and lies left
  of $N$. Together with $t'NpNs'$ it forms a simple cycle $\cgD$.
  The arc lies in the left region of $N$ containing the subpath
  $t'NpNs'$. By the Jordan curve theorem, $\cgD$ bounds the part
  of this region incident with that subpath. In particular, the sector
  at $p$ left of $N$ is inside $\cgD$.
  The two boundary paths from $s'$ to $p$ are witnessing paths, so every
  boundary vertex $q$ satisfies $y_\cgB\le q\le p$ in $P$.
  Now note that the path $p \Wleft(z) z$ starts inside $\cgD$ (since $e^+$ is left of $N$). 
  So if $z$ were outside $cgD$, then when the path $p \Wleft(z) z$ intersected the boundary of $\cgD$, a directed cycle would be created.
  Hence $z$ is inside $\cgD$.
  On the other hand, $v$ is outside $\cgD$. Indeed, $x_0$ is outside,
  and a witnessing path from $x_0$ to a vertex inside $\cgD$ must meet
  a boundary vertex which is up from $y_\cgB$; this is impossible for
  $v\parallel y_\cgB$.
  Finally, a witnessing path from $v$ to $z$ crosses the boundary at
  some $q\le p\le b$, forcing $v<b$ in $P$.
  This contradiction completes the proof of Proposition~\ref{pro:tilt-one-direction}.
  \end{proof}

For each $\theta\in\{0,1\}$, we will now define a binary
relation $\NTR_\theta$ as follows:
\begin{enumerate}
  \item we assign to $\NTR_\theta$ pairs 
  $(x_0,u)$ for all $u\neq x_0$ in $P$; and 
  \item we assign to $\NTR_\theta$ all pairs $(a,b)$ such that 
  $a,b\in U_P(x_0)$, $\sd(a)\equiv\theta\mod2$, and $a<_P b$
\emph{except} those pairs $(a,b)$ where the comparability $a<_P b$ tilts right.
\end{enumerate}
Note that the abbreviation $\NTR$ in this notation is short for ``not tilting right.''
We also define for each $\theta\in\set{0,1}$ in the obvious symmetric manner a second binary relation denoted
$\NTL_\theta$: 
\begin{enumerate}
  \item we assign to $\NTL_\theta$ pairs 
  $(x_0,u)$ for all $u\neq x_0$ in $P$; and 
  \item we assign to $\NTL_\theta$ all pairs $(a,b)$ such that 
  $a,b\in U_P(x_0)$, $\sd(a)\equiv\theta\mod2$, and $a<_P b$
\emph{except} those pairs $(a,b)$ where the comparability $a<_P b$ tilts left.
\end{enumerate}

\begin{proposition}\label{pro:NTR-transitive}
  For each $\theta\in\{0,1\}$, the binary relations $\NTR_\theta$ and
  $\NTL_\theta$ are strict partial orders, i.e., each of $\NTR_\theta$ and 
  $\NTL_\theta$ is irreflexive and transitive.
\end{proposition}
\begin{proof}
  We give the argument to show that $\NTR_\theta$ is a strict partial order. 
  The argument
  for $\NTL_\theta$ is symmetric.  We first note that $\NTR_\theta$ is irreflexive
  since $a<_P b$ whenever $(a,b)\in\NTR_\theta$.  It remains only to show
  that $\NTR_\theta$ is transitive.

  Let $(a,b),(b,c)\in\NTR_\theta$.
  If $a=x_0$, then $(a,c)\in\NTR_\theta$ by definition, so assume $a\ne x_0$. 
  Then, $a<b< c$ in $P$, so
  $a< c$ in $P$. 
  It follows that $(a,c)\in\NTR_{\theta}$ unless the comparability $a<_P c$
  tilts right. 
  We assume that this is the case and argue to a contradiction.

  Let $u$ be an element of $P$ witnessing that the comparability $a<_P c$ tilts right. Since $u\parallel_P \{a,c\}$, 
  and $a<_P b<_P c$, we know $u\parallel_P b$.  
  Let $(j,\cgB)$ be the address of the dangerous pair $(u,c)$. 
  Then, we know that $a\in A_{\cgB}$, $c\in B_{\cgB}$. 
  Since $a<c$ in $P$, Proposition~\ref{pro:block-details}.\ref{pro:item:sda-more-than-j} implies that 
  $\sd(a)=j$. 
  Since $(a,b)\in \NTR_{\theta}$, we conclude that $j\equiv\theta\mod2$. 
  Since $a<_P b$ and $a\in A_\cgB$, we know that $b\in A_\cgB\cup B_\cgB$, 
  so we have identified two cases.

  \smallskip
  \noindent
  \textit{Case 1.}\quad $b\in A_{\cgB}$.

We claim that the element $u$ evidences that
  the comparability $b<_P c$ tilts right. 
  Since we already know that $(u,c)$ is dangerous and 
  $b\in A_{\cgB}$, all we need to show is that
  $(u,b)$ is a right pair. We have already seen that $u\parallel_P b$.

  Let $s$ be the largest element of $P$ common to $\Wleft(a)$ and $\Wleft(u)$. 
  Let $e^-$ and $e^+$ be the edges of $\Wleft(u)$ immediately before and immediately after $s$. Again, in case $s=x_0$ we set $e^-=e_{-\infty}$. 
  Let $e=ss'$ be the edge of $\Wleft(a)$ immediately after $s$. 
  Since $(u,a)$ is a right pair, $e^+$ is right of $e$ in the 
  $(s,e^-)$-ordering. 
  Note that $s < s' \leq a < b$ in $P$. 
  Thus, Proposition~\ref{pro:piotrek} implies that $\Wleft(b)$ is $x_0$-left of $\Wleft(u)$. 

  Therefore $(u,b)$ is either a right pair or an inside pair.
  To exclude the latter, note that $\sd(u)=j$ since $(u,c)$ is dangerous,
  while $\sd(b)=j$ follows from $b<_P c$ and
  Proposition~\ref{pro:block-details}.\ref{pro:item:sda-more-than-j}.
  Their lower shadows agree, and their $j$-shadows differ since their
  terminal elements $u,b$ are distinct.  Thus $(u,b)$ has depth $j$.
  If it were inside, Proposition~\ref{pro:geometry-inside} would give
  $u\in\shad_j(b)$.  Proposition~\ref{pro:block-details}.\ref{pro:item:incomp}
  would then make $u$ incomparable to every element of $B_\cgB$.
  But a neighbor $(w,z)$ of the dangerous pair $(u,c)$ has
  $z\in B_\cgB$ and $u<_P z$, a contradiction.
  Hence $(u,b)$ is a right pair and witnesses that $b<_P c$ tilts right,
  contrary to $(b,c)\in\NTR_\theta$.

  \smallskip
  \noindent
  \textit{Case 2.}\quad $b\in B_{\cgB}$.

  In this case, we claim that the element $u$ evidences that 
  $a<_P b$ tilts right. 
  Since $u\parallel_P b$, by Proposition~\ref{pro:all-inc-are-inside}, $(u,b)$ is an inside pair with address $(j,\mathcal{B})$.
  All we need to prove is that $(u,b)$ is a dangerous pair. 
  Note that, by Proposition~\ref{pro:shadow-comp}.\ref{pro:item:depth-of-u}, $b\in B_{\cgB}$ implies $\sd(b)\ge j+1$. 
  Since $(b,c)\in\NTR_{\theta}$, we get $\sd(b)\equiv\theta\mod2$ which implies
  $\sd(b)\ge j+2$. 
  Since $b<c$ in $P$, we can apply Corollary~\ref{cor:b-and-c-have-the-same-shadow} and we obtain that $\shad_{j+1}(b)=\shad_{j+1}(c)$.
  This combined with 
  Proposition~\ref{pro:j+1-dangerous-equiv} implies that
  $(u,b)$ is a dangerous pair. 
  Thus, $a<_P b$ tilts right.
  The contradiction completes the proof.
  \end{proof}

Note that all four strict orders $\NTL_0$, $\NTL_1$, $\NTR_0$, and $\NTR_1$ are strict suborders of $P$. 
Therefore, $D_0$ and $D_1$ are incomparable pairs in all four of them.

\begin{lemma}\label{lem:lo-NTR+NTL}
  Let $\theta\in\{0,1\}$. Then $D_{\theta}$ is reversible in $\NTR_{\theta}$ and $D_{\theta}$ is reversible in $\NTL_{\theta}$.
%
\end{lemma}
\begin{proof}
  Fix $\theta\in\{0,1\}$.
  We show that $D_{\theta}$ is reversible in $\NTR_{\theta}$. 
  The argument for $\NTL_{\theta}$ is symmetric. 

  Suppose to the contrary that $D_{\theta}$ is not reversible in $\NTR_{\theta}$. 
  Let $((a_1,b_1),\dots,(a_k,b_k))$ be an alternating cycle in $\NTR_{\theta}$ with $(a_\alpha,b_\alpha)\in D_\theta$ for all $\alpha\in[k]$. 
  Therefore, for all $\alpha\in[k]$, either 
  $a_\alpha=b_{\alpha+1}$ or $(a_\alpha,b_{\alpha+1})\in\NTR_\theta$.
  Of all such cycles, we choose one for which $k$ is minimum. 
  Note that this automatically implies that the alternating cycle is strict.

  \begin{claim}\label{clm:NTR-1}
    There is no $\alpha\in[k]$ such that $(a_{\alpha+1},b_\alpha)\in D_\theta$.
  \end{claim}
  \begin{proof}
    The claim is clearly true when $k=2$, so suppose $k\geq 3$.
    If the claim fails, then let $(a_{\alpha+1},b_\alpha)\in D_\theta$ and consider
    \[
    (\ldots,(a_{\alpha-1},b_{\alpha-1}), (a_{\alpha+1},b_{\alpha}),(a_{\alpha+2},b_{\alpha+2}),\ldots).
    \]
    Thus, it is an alternating cycle in $\NTR_{\theta}$ of size $k-1$ with all pairs in $D_{\theta}$. 
    This contradicts the assumption that $k$ is minimum possible.
  \end{proof}
  \begin{claim}\label{clm:NTR-2}
    For all $\alpha\in[k]$, $(a_\alpha,b_{\alpha+1})\in\NTR_\theta$ (so $a_{\alpha}\neq b_{\alpha+1}$).
  \end{claim}
  \begin{proof}
    If the claim fails, then let $a_{\alpha}=b_{\alpha+1}$. 
    Since $(a_{\alpha+1},b_{\alpha+1}) \in D_{\theta}$ and 
    $(a_{\alpha},b_{\alpha})\in D_{\theta}$ and by transitivity of $D_{\theta}$ (see Proposition~\ref{pro:dangerous-transitive}), we conclude 
    $(a_{\alpha+1},b_{\alpha})\in D_{\theta}$. 
    This contradicts Claim~\ref{clm:NTR-1}.
  \end{proof}

  For each $\alpha\in[k]$, let $(j_\alpha,\cgB_\alpha)$ be the address of
  the dangerous pair $(a_\alpha,b_\alpha)$. We let
  $A_\alpha=A_{\cgB_\alpha}$ and $B_{\alpha}=B_{\cgB_\alpha}$.
  Then $a_\alpha\in A_\alpha$ and $b_\alpha\in B_\alpha$. 

  \begin{claim}\label{clm:NTR-3}
    For all $\alpha\in[k]$, $j_\alpha\le j_{\alpha+1}$.
  \end{claim}
  \begin{proof}
    Since $a_{\alpha}< b_{\alpha+1}$ in $\NTR_{\theta}$ 
    we conclude that $a_{\alpha}< b_{\alpha+1}$ in $P$. 
    This forces $b_{\alpha+1}\in A_\alpha\cup B_\alpha$.
    Suppose to the contrary that $j_{\alpha}> j_{\alpha+1}$. 
    Let $j=j_{\alpha+1}$. 

    Since $j_{\alpha}\equiv j_{\alpha+1}\equiv \theta \mod 2$, 
    we conclude that $j=j_{\alpha+1}\leq j_{\alpha}-2$. 
    Since $b_{\alpha}$ and $b_{\alpha+1}$ are together in $\cgB_{\alpha}$, 
    we know that $\sd(b_{\alpha}),\sd(b_{\alpha+1})\geq j_{\alpha} \geq j+2$ and 
    for all $0\leq i \leq j+1$, 
    we have $\shad_i(b_{\alpha})=\shad_{i}(b_{\alpha+1})$.

    Now consider the block $\cgB_{\alpha+1}$. 
    Note that $a_{\alpha+1}\in A_{\alpha+1}$, $b_{\alpha}, b_{\alpha+1}\in B_{\alpha+1}$. 
    Proposition~\ref{pro:j+1-dangerous-equiv} implies that 
    $(a_{\alpha+1},b_{\alpha}) \in D_{\theta}$. 
    Again, this 
    contradicts Claim~\ref{clm:NTR-1}.
  \end{proof}

  The last claim implies that there is an integer $j$ such that $j_\alpha=j$ for
  all $\alpha\in[k]$.  We have already noted that for each $\alpha\in[k]$,
  we know that $b_{\alpha+1}\in A_\alpha\cup B_\alpha$.
  \begin{claim}\label{clm:NT-4}
    For all $\alpha\in[k]$, $\cgB_{\alpha+1}=\cgB_\alpha$ if
    $b_{\alpha+1}\in B_\alpha$.  Furthermore,
    $\cgB_{\alpha+1}\subsetneq \cgB_\alpha$ if $b_{\alpha+1}\in A_\alpha$.
  \end{claim}
  \begin{proof}
    If $b_{\alpha+1}\in B_\alpha$, then $\cgB_{\alpha}$ is the terminal block of the $j$-shadow of $b_{\alpha+1}$. 
    Therefore, $\shad_j(b_{\alpha})=\shad_{j}(b_{\alpha+1})$ and 
    in particular $\cgB_\alpha=\cgB_{\alpha+1}$.

    Now suppose that $b_{\alpha+1}\in A_\alpha$.
    Thus $b_{\alpha+1}\parallel y_{\cgB_{\alpha}}$. 
    By Proposition~\ref{pro:shadow-comp}, 
    we know that the terminal block of $\shad_j(b_{\alpha+1})$ is strictly contained in 
    $\cgB_\alpha$, as desired.
  \end{proof}

  The preceding claim implies that there is a shadow block $\cgB$ such that
  $\cgB_\alpha=\cgB$, for all $\alpha\in[k]$.  
  Thus, for all $\alpha\in[k]$, the pair $(a_{\alpha},b_{\alpha})$ has address $(j,\cgB)$ and $a_\alpha\in A_\cgB$, $b_{\alpha}\in B_\cgB$. 
  Moreover, $\sd(a_{\alpha})=j$
  for each $\alpha\in[k]$.

  \begin{claim}\label{clm:NT-5}
    $\{a_1,\dots,a_k\}$ is an antichain in $P$.
  \end{claim}
\begin{proof}
    The set $\{a_1,\dots,a_k\}$ is an antichain in $\NTR_\theta$, since
    the chosen alternating cycle is strict in that order. 
    All the $a_i$'s have shadow depth $j$.  As observed before
    Proposition~\ref{pro:tilt-one-direction}, any comparability in $P$ between
    them cannot tilt. Therefore, if $a_{\alpha}<a_{\beta}$ in $P$, 
    then $a_{\alpha}<a_{\beta}$ in $\NTR_{\theta}$. 
    This proves that $\{a_1,\dots,a_k\}$ is an antichain in $P$.
\end{proof}
 Since each element of $\{a_1,\dots,a_k\}$ is comparable
  with an element of $B_\cgB$, by 
    Proposition~\ref{pro:block-details}.\ref{pro:item:incomp}, there cannot be distinct integers
  $\alpha,\beta\in[k]$ such that $(a_\alpha,a_\beta)$ is an inside pair. After a
  relabeling we may assume that $(a_\alpha,a_1)$ is a right pair in $P$
  for every $\alpha\in[k]$ with $\alpha\ge2$.  
  However, this implies
  that the comparability $a_1< b_2$ in $P$ tilts right 
  as is witnessed by a dangerous pair $(a_2,b_2)$. 
  The contradiction completes the 
  proof of the lemma.
\end{proof}

Before defining the remaining four linear orders, we prove an elementary abstract lemma that 
allows us to reduce the size of the eventual realizer (from $13$ in the first version of the paper to $10$).


\begin{completionlemma}
  Let $P$ be a poset, let $X\subseteq P$, and let
  $D\subseteq\Inc(P)\cap(X\times P)$.
  Suppose that $R$ and $S$ are linear orders on the elements of $P$ such that
  \begin{enumerate}
    \item $R$ reverses $D$ and $S$ reverses $D$; and
    \item for all $a,b\in P$
  \[
  a\in X,\ a<b\ \textrm{in $P$}\quad\Longrightarrow\quad \textrm{($a<b$ in $R$ or $a<b$ in $S$).}
  \]
  \end{enumerate}
  Then there is a linear order $T$ on the elements of $P$ such that
  \begin{enumerate}
    \item $T$ reverses $D$; and
    \item for all $a,b\in P$
  \[
  a<b\ \textrm{in $P$}\quad\Longrightarrow\quad \textrm{($a<b$ in $S$ or $a<b$ in $T$).}
  \]
  \end{enumerate}
\end{completionlemma}

\begin{proof}
  Define a strict partial order $Q$ on the elements of $P$ by
  \[
    a< b\textrm{ in $Q$}\quad\Longleftrightarrow\quad\textrm{($a<b$ in $P$ and $b< a$ in $S$).}
  \]
  Indeed, $Q$ is a  strict partial order as it is the intersection of 
  the strict order of $P$ and a strict linear order.
  Since the relation of $Q$ is a subset of the relation of $P$, 
  we have that $D$ is a set of incomparable pairs in $Q$. 

  We claim that $D$ is reversible in $Q$. 
  Suppose to the contrary that $((a_1,b_1),\ldots,(a_k,b_k))$ is an
  alternating cycle in $Q$ with all pairs in $D$, with subscripts
  interpreted cyclically. 
  Thus $a_i\le b_{i+1}$ in $Q$ for every $i\in[k]$.
  If $a_i=b_{i+1}$, then $a_i\le b_{i+1}$ in $R$.
  Otherwise $a_i< b_{i+1}$ in $P$ and $b_{i+1}< a_i$ in $S$. 
  Since $a_i\in X$,
  the hypothesis forces $a_i< b_{i+1}$ in $R$.
  As $R$ reverses $D$, we therefore have
  \[
    b_i< a_i\le b_{i+1}\ \textrm{in $R$}\qquad\textrm{for all }i\in[k],
  \]
  a contradiction. This proves the claim.

  Choose a linear extension $T$ of $Q$ reversing $D$. 
  Consider $a$, $b$ in $P$ with $a<b$ in $P$. 
  Obviously one of the two is true: $a< b$ in $S$, or $b< a$ in $S$. 
  Note that if $b<a$ in $S$ then we have $a<b$ in $Q$ so also $a<b$ in $T$. 
  This completes the proof of the lemma.
\end{proof}

For each $\theta\in\{0,1\}$, we proceed as follows. 
Let 
\[
  X_\theta=\{a\in P:\sd(a)\equiv\theta\mod2\},
\]
where, as before, $\sd(x_0)=-1$.
By Lemma~\ref{lem:lo-NTR+NTL}, choose a linear extension $R_\theta$
of $\NTR_\theta$ and a linear extension $S_\theta$ of $\NTL_\theta$,
both reversing $D_\theta$.
If $a<b$ in $P$ with $a\in X_\theta\setminus\{x_0\}$, then
Proposition~\ref{pro:tilt-one-direction} shows that $(a,b)$ belongs to
at least one of $\NTR_\theta$ and $\NTL_\theta$.
Thus $a<b$ in at least one of $R_\theta$ and $S_\theta$.
If $a=x_0$, then $(x_0,b)$ belongs to both suborders by definition,
so the same conclusion holds.

Since $D_\theta\subseteq\Inc(P)\cap(X_\theta\times P)$, the Completion
Lemma gives a linear order $T_\theta$ such that $T_\theta$ and
$S_\theta$ both reverse $D_\theta$ and every comparability of $P$
remains a comparability in at least one of $S_{\theta}$, $T_{\theta}$.
We now define the remaining linear orders:
\[
  L_7=S_0,\qquad L_8=T_0,\qquad L_9=S_1,\qquad L_{10}=T_1.
\]
The orders $R_0$ and $R_1$ are used only to establish the existence
of $T_0$ and $T_1$; they are not included in the Boolean realizer.

To complete the proof of Theorem~\ref{thm:bdim-10}, use the Boolean formula
\[
  \phi(t_1,\ldots,t_{10})=
  \left(\bigwedge_{i=1}^{6}t_i\right)\wedge
  (t_7\vee t_8)\wedge(t_9\vee t_{10}).
\]
We show that, for all $a,b\in P$, setting
$t_i=[a\le b\text{ in }L_i]$ gives $\phi(t_1,\ldots,t_{10})=1$ 
if and only if $a\le b$ in $P$.

If $a=b$, then all ten bits are one, so the formula returns one.
If $a<b$ in $P$, then $t_1,\ldots,t_6$ are all one because the first
six orders are linear extensions of $P$. For each $\theta\in\{0,1\}$,
at least one of $T_\theta$ and $S_\theta$ puts $a$ before $b$.
Thus both remaining disjunctions are one, and the formula returns one.
This includes $a=x_0$; indeed, both $S_0$ and $S_1$ put $x_0$
before every other element.

Now suppose $a\not\le b$ in $P$. If one of $t_1,\ldots,t_6$ is zero,
the formula returns zero. Otherwise $(a,b)$ is dangerous, so it
belongs to $D_\theta$ for some $\theta\in\{0,1\}$.
Both $S_\theta$ and $T_\theta$ reverse it. Consequently $t_7\vee t_8=0$
when $\theta=0$, and $t_9\vee t_{10}=0$ when $\theta=1$.
In either case the formula returns zero.
Thus $(L_1,\ldots,L_{10})$, with the same formula $\phi$ for every
such poset $P$, is a Boolean realizer of size $10$.
This completes the proof of Theorem~\ref{thm:bdim-10}.
\section{Dimension and Standard Example Number}\label{sec:dim-bounded}

In this section, we give the second part of the proof of 
Theorem~\ref{thm:dim-boundedness}, i.e., we show that if $P$ is a poset 
with a unique minimal element and a planar cover graph, then 
$\dim(P)\le 2\se(P)+2$.  The initial set up is the same as in the last 
two sections.  We assume that $P$ is a poset with a planar cover graph and a 
unique minimal element denoted $x_0$.  We fix a plane drawing of the cover graph 
$G$ of $P$ with $x_0$ on the exterior face.  

In Section~\ref{sec:common}, we showed that
there are $2$ subsets of $\Inc(P)$ that are reversible in $P$ and 
cover all incomparable pairs in $P$ except the \emph{inside}
pairs.  Recall that for $\theta\in\{0,1\}$, we let $I_\theta$ consist of all 
inside pairs $(a,b)$ such that if $(j,\cgB)$ is the address of $(a,b)$, then 
$j\equiv\theta\mod2$.  Then we have the partition $I_0\sqcup I_1$ of the
set of all inside pairs.   To complete the proof of
Theorem~\ref{thm:dim-boundedness}, we show that for each $\theta\in\{0,1\}$, the
set $I_\theta$ can be covered with $\se(P)$ sets which are reversible in $P$.


We need a preliminary result that is in the spirit of Proposition~\ref{pro:force-sides}.
\begin{proposition}\label{pro:force-sides-new}
  Let $\theta\in\{0,1\}$, and let $((a,b),(a',b'))$ be a strict alternating
  cycle in $P$ of pairs from $I_\theta$ with $(a,a')$ a left pair.  Let $(j,\cgB)$ be
  the common address of $(a,b)$ and $(a',b')$, and let $q$ be an element of
  $B_\cgB$ such that $(q,b)$ is a left pair.  If $N$ is a separating path associated
  with $a<_P b'$, then $q$ is right of $N$.
\end{proposition}
\begin{proof}
  From Proposition~\ref{pro:sac-comprehensive}, we know that $(b,b')$ is a left pair.
  Let $p$ be the peak of $N$, and let $z$ be the largest element of $P$ common to
  $\Wleft(q)$ and $\Wleft(p)$. By Proposition~\ref{pro:shadow-block-maxB}, $y_{\cgB}\leq z \leq p$. 
  Note that $z\in \Wleft(b)$ as $(q,b)$ and $(b,b')$ are left pairs.
  This implies $z< p$ in $P$ as $z=p$ would give $a\leq p = z \leq b$ in $P$. 
  Let $e^-$, $e^+$ be the edges in $\Wleft(q)$ immediately before and immediately after $z$. 
  Let $e$ be the edge of $\Wleft(p)$ immediately after $z$. 
  Since $(q,b')$ is a left pair by Proposition~\ref{pro:details-on-LR}.\ref{pro:item:left-pair-is-transitive}, we have that 
  $e^+$ is left of $e$ in the $(z,e^-)$-ordering. 
  Since $e$ lies in $N$, we conclude that $e^+$ is right of $N$. 
  If the element
  $q$ is not right of $N$, then $z\Wleft(q)q$ must contain
  an element $w$ from $N$ with $z\neq w$.  
  Since $\Wleft(p)$ and $\Wleft(q)$ are $x_0$-consistent,
  the element $w$ does not belong to $\Wleft(p)$. 
  Therefore, it belongs to $x_{\cgB}Np$ with
  $w\neq p$. 
  Now, the witnessing path $x_0\Wleft(w)wNp$ 
  provides a witnessing path to $p$ whose edge at the first divergence
  precedes the chosen edge of $\Wleft(p)$. 
  This contradicts the greedy construction of $\Wleft(p)$. 
  The contradiction completes the proof.
\end{proof}

For each $\theta\in\{0,1\}$, we define a directed graph $H_\theta$ whose
vertex set is $I_\theta$.  In $H_\theta$, there are two kinds of edges.  The first
kind will be called a \emph{strong} edge.  When $(a,b)$ and $(a',b')$
are vertices from $I_\theta$, we have a strong edge in $H_\theta$
from $(a,b)$ to $(a',b')$ when $a< b'$ in $P$, $a'< b$ in $P$, and 
$(a,a')$ is a left pair.
We then have a \emph{weak} edge from $(a,b)$ to $(a',b')$ in $H_\theta$ when there
is a pair $(u,v)\in I_\theta$ such that $((a,b),(a',b'),(u,v))$ is
a strict alternating cycle in $P$, $(a,a')$ is a left pair, and $(a',u)$
is a left pair.  In discussing weak edges, we will refer to the
pair $(u,v)$ as a \emph{helper} pair for edge $((a,b),(a',b'))$.

Every edge of $H_\theta$ advances in the left-pair relation on its
first coordinates.  This relation is transitive and irreflexive by
Proposition~\ref{pro:details-on-LR}; hence $H_\theta$ is acyclic. When $n\ge1$, we will say that a sequence $((a_1,b_1),\dots,(a_n,b_n))$ is a 
directed path in $H_\theta$ of size~$n$ when there is a directed edge
from $(a_i,b_i)$ to $(a_{i+1},b_{i+1})$ whenever $1\le i<n$.  Also, we will
say that this directed path \emph{starts at $(a_1,b_1)$}.  Note that when
$n\ge2$, and $1\le i<n$, $(a_i,a_{i+1})$ will be a left pair, independent of 
whether the edge from $(a_i,b_i)$ to $(a_{i+1},b_{i+1})$ is strong or weak.

When $(a,b)\in I_\theta$, we consider a trivial sequence $((a,b))$ as a
directed path of size~$1$.  For a positive integer $m$, we let $I_\theta(m)$
consist of all pairs $(a,b)\in I_\theta$ such that the maximum size of a
directed path in $H_\theta$ starting at $(a,b)$ is $m$.

\begin{lemma}\label{lem:dim(J(m))} 
  For each $\theta\in\{0,1\}$ and each positive integer $m$, the set
  $I_\theta(m)$ is reversible in $P$.
\end{lemma}
\begin{proof}
  Fix $\theta\in\set{0,1}$ and an integer $m\ge1$. 
  We then assume that $((a_1,b_1),\dots,(a_k,b_k))$
  is a strict alternating cycle in $P$ of pairs from $I_\theta(m)$ and argue to
  a contradiction.  

  Using Lemma~\ref{lem:cgI-comprehensive}, we know that there
  is a pair $(j,\cgB)$ that is the common address of all the pairs on the cycle. 
  We also know that for every $\alpha\in[k]$,
  $a_\alpha\in A_\cgB$ and $b_\alpha\in B_\cgB$. Consider $\alpha,\beta\in[k]$ with $\beta\neq\alpha+1$. Then $a_\alpha \parallel b_\beta$ in $P$ because we are working in a strict alternating cycle.
  It follows by Proposition~\ref{pro:all-inc-are-inside} that  $(a_\alpha,b_\beta)$ is a pair in $I_\theta$ and it
  has address $(j,\cgB)$. 

  Proposition~\ref{pro:sac-comprehensive} implies that $(a_{\alpha},a_{\beta})$ is a left pair or a right pair for all $\alpha\neq \beta$ in $[k]$. 
  Therefore we can assume that $a_1$ is the leftmost, i.e.\ 
  $(a_1,a_\alpha)$ is a left pair, whenever $2\le\alpha\le k$.
  Since $(a_k,b_2)\in I_{\theta}$, we have a strong edge 
  from $(a_1,b_1)$ to $(a_k,b_2)$ in $H_{\theta}$.
 Therefore, $m\ge2$.

  Now suppose $k=2$.  Then there is a strong directed edge in $H_\theta$ from
  $(a_1,b_1)$ to $(a_2,b_2)$.  Furthermore, because $H_\theta$ is acyclic, a directed path of size~$m$
  starting at $(a_2,b_2)$ can be extended to a directed path of size~$m+1$
  starting at $(a_1,b_1)$ simply by prepending $(a_1,b_1)$ at the beginning. 
  This implies that $(a_1,b_1)\not\in I_{\theta}(m)$. 
  The contradiction shows that $k\ge3$.

  Consider the $3$-element antichain $\{a_1,a_2,a_k\}$ and the strict alternating cycle $((a_1,b_1),(a_2,b_2),(a_k,b_3))$ in $P$.  
  We know that $(a_1,a_2)$ and $(a_1,a_k)$ are left pairs.  We also know
  that one of $(a_2,a_k)$ and $(a_k,a_2)$ is a left pair.

  \smallskip
  \noindent
  \textit{Case 1.}\quad $(a_2,a_k)$ is a left pair.

  \smallskip

  In this case, the strict alternating cycle $((a_1,b_1),(a_2,b_2),(a_k,b_3))$ produces a weak directed edge in $H_\theta$ from $(a_1,b_1)$ to
  $(a_2,b_2)$, with the pair $(a_k,b_3)$ being the helper pair.  This implies that 
  a directed path of size~$m$ starting at $(a_2,b_2)$ can be extended to a 
  directed path of size~$m+1$ starting at $(a_1,b_1)$ simply by prepending 
  $(a_1,b_1)$ at the front.  This implies $(a_1,b_1)\not\in I_{\theta}(m)$, 
  a contradiction.

  \smallskip
  \noindent
  \textit{Case 2.}\quad $(a_k,a_2)$ is a left pair.

  \smallskip

  From the strict alternating cycle $((a_1,b_1),(a_2,b_2),(a_k,b_3))$, Proposition~\ref{pro:sac-comprehensive} shows that $(b_3,b_1)$ and $(b_1,b_2)$ are left pairs. 
  Let $((w_1,z_1),\dots,(w_m,z_m))$ be a directed path of
  size~$m$ in $H_\theta$ with $(w_1,z_1)=(a_2,b_2)$.  Then consider the 
  sequence 
  \[ 
  ((a_1,b_1),(a_k,b_2),(w_2,z_2),\dots,(w_m,z_m)) 
  \]
  of pairs from $I_\theta$. (Recall that $m\geq2$ so $(w_2,z_2)$ exists.)  
  Clearly, this sequence has size $m+1$, and it 
  starts with $(a_1,b_1)$.  
  We make the following observations: 
  (1) there is a strong directed
  edge in $H_\theta$ from $(a_1,b_1)$ to $(a_k,b_2)$; and 
  (2)~if $2\le i<m$, there
  is a directed edge in $H_\theta$ from $(w_i,z_i)$ to 
  $(w_{i+1},z_{i+1})$.  Accordingly, the only edge missing so far in $H_{\theta}$ 
  to complete the directed path on our sequence is from $(a_k,b_2)$ to $(w_2,z_2)$.   

  Recall that $(a_k,a_2)$ and $(a_2,w_2)$ are left pairs, so
  $(a_k,w_2)$ is a left pair.

  We split the argument again depending on the type of edge in $H_{\theta}$
  from $(w_1,z_1)$ to $(w_2,z_2)$. 

  First, we assume that $((w_1,z_1),(w_2,z_2))$ is strong.
  Then $w_2<_P z_1=b_2$, so it suffices to prove $a_k<_P z_2$. 
  Let $N$ be a separating path in the shadow block $\cgB$ 
  associated with the comparability $a_2=w_1<_P z_2$. 
  Since $(a_k,a_2)$ is a left pair, 
  Proposition~\ref{pro:force-sides} implies $a_k$ is left of $N$.
  We know that $(b_1,b_2)$ is a left pair and $b_2=z_1$. For the strict alternating cycle $((w_1, z_1), (w_2, z_2))$ in $P$, Proposition~\ref{pro:sac-comprehensive} shows that $(z_1,z_2)$ is a left pair.
  Thus $(b_1,z_2)$ is a left pair.  Since $b_1\parallel a_2$ in $P$,
  Proposition~\ref{pro:force-sides} implies that $b_1$ is right of $N$.

  For the comparability $a_k< b_1$ in $P$, we now have $a_k$ left of $N$, and $b_1$
  right of $N$.  Proposition~\ref{pro:force-comp} then implies that $a_k< z_2$ in $P$.
  This comparability completes requirements for a strong directed edge from $(a_k,b_2)$ to $(w_2,z_2)$ in $H_{\theta}$.
  Therefore, we have completed a directed path in $H_{\theta}$ from $(a_1,b_1)$ of length $m+1$. 
  This contradicts the fact that $(a_1,b_1)\in I_{\theta}(m)$.

  Now, we assume that $((w_1,z_1),(w_2,z_2))$ is weak.

  Then let $((w_1,z_1),(w_2,z_2),(u,v))$ be a strict alternating cycle in $P$
  evidencing that this is a weak edge in $H_\theta$.
  Then, again by Lemma~\ref{lem:cgI-comprehensive}, 
  $(u,v)$ is an inside pair whose address is $(j,\cgB)$, $u\in A_{\cgB}$, 
  and $v\in B_{\cgB}$.  
  Also, $(w_1,w_2)$ and $(w_2,u)$ are left pairs.

  We note that $(b_3,b_1)$ and $(b_1,b_2)$ are left pairs.
  Therefore, $(b_3,b_2)$ is a left pair.
  We apply Proposition~\ref{pro:force-sides-new} for the strict alternating
  cycle $((w_2,z_1),(u,v))$, the separating path $N$ associated with the comparability
  $w_2<_P v$, and the element $b_3$.  The proposition implies that $b_3$ is right of $N$.
  On the other hand, since $(a_2,w_2)=(w_1,w_2)$ is a left pair, Proposition~\ref{pro:force-sides}
  implies $a_2$ is left of $N$.  Using Proposition~\ref{pro:force-comp} for
  the comparability $a_2<_P b_3$, we conclude that $w_1=a_2<_P v$, which is
  false. The contradiction completes the proof.
\end{proof}

\begin{proposition}\label{pro:strong-transitive}
  Let $\theta\in\{0,1\}$, 
  and let
  $((a_1,b_1),(a_2,b_2))$, $((a_2,b_2),(a_3,b_3))$ be strong edges in $H_\theta$.  
  Then $((a_1,b_1),(a_3,b_3))$ is a strong edge in $H_{\theta}$ as well.
\end{proposition}
\begin{proof}
  Since $(a_1,a_2)$ and $(a_2,a_3)$ are
  left pairs, we know that $(a_1,a_3)$ is a left pair. 
  It suffices
  to show that $a_1< b_3$ in $P$ and $a_3< b_1$ in $P$. 
  We give the argument to show
  that $a_1< b_3$ in $P$.  The argument for $a_3< b_1$ in $P$ is symmetric.

  By Lemma~\ref{lem:cgI-comprehensive}, we can fix $(j,\cgB)$ to be 
  the common address of the three pairs on the path.  Then
  let $N$ be a separating path in the shadow block $\cgB$ associated with
  the comparability $a_2< b_3$ in $P$.  Since $(a_1,a_2)$ is a left pair, 
  Proposition~\ref{pro:force-sides} implies that $a_1$ is left of $N$.
  Since $(b_2,b_3)$ is a left pair, and $b_2\parallel a_2$ in $P$, 
  Proposition~\ref{pro:force-sides} also implies that $b_2$ is right of $N$.
  Since $a_1$ is left of $N$, $b_2$ is right of $N$, and $a_1< b_2$ in $P$,
  Proposition~\ref{pro:force-comp} implies $a_1< b_3$ in $P$. With this observation,
  the proof is complete.
\end{proof}
Let $n\ge2$, and let $((a_1,b_1),\dots,(a_n,b_n))$ be a directed
path in $H_\theta$ with each edge $((a_i,b_i),(a_{i+1},b_{i+1}))$ being strong 
for $i\in[n-1]$.  Proposition~\ref{pro:strong-transitive} implies
that the points in $\{a_1,\dots,a_n,b_1,\dots,b_n\}$ determine a copy of
the standard example $S_n$.  

Now we are missing just one final piece to complete the proof.

\begin{proposition}\label{pro:path-bound}
  For each $\theta\in\{0,1\}$,
  $H_{\theta}$ has no directed path on more than $\se(P)$ vertices.
\end{proposition}
\begin{proof}
  Let $\theta\in\set{0,1}$ and
  let $((a_1,b_1),\ldots,(a_n,b_n))$ be a directed path in $H_{\theta}$.
  If $n=1$, the conclusion follows from $\se(P)\ge1$.  Assume $n\ge2$. 
  Let $(j,\cgB)$ be the common address of all pairs (and helper pairs) in the path 
  (it exists by Lemma~\ref{lem:cgI-comprehensive}).
  We show that $H_\theta$ contains a directed path of the same length 
  with all edges being strong. 
  This will determine a copy of $S_n$ in $P$.
  Therefore, $n\le\se(P)$, as desired.

  If the directed edge from $(a_1,b_1)$ to $(a_2,b_2)$ is strong,
  we set $(u_1,v_1)=(a_1,b_1)$ and $(u_2,v_2)=(a_2,b_2)$.
  If the directed edge from $(a_1,b_1)$ to $(a_2,b_2)$ is weak, and
  is evidenced by the strict alternating cycle $((a_1,b_1),(a_2,b_2),(u,v))$ in $P$,
  we set
  $(u_1,v_1)=(a_1,v)$ and $(u_2,v_2)=(a_2,b_2)$. 
  Recall that $a_1\in A_{\cgB}$ and $v\in B_{\cgB}$, 
  therefore by Proposition~\ref{pro:all-inc-are-inside} 
  $(a_1,v)$ is a vertex in $H_{\theta}$.

  In both cases, the
  following statements hold when $s=2$:
  \begin{enumerate}
    \item $(u_s,v_s)=(a_s,b_s)$.
    \item $((u_1,v_1),\dots,(u_s,v_s))$ is a strong directed path of size~$s$ in 
      $H_\theta$.
  \end{enumerate}

  Now we continue the construction as long as  $2\le s< n$ 
  and we will keep the two items above as invariants. 

  If the directed edge from $(a_s,b_s)$ to $(a_{s+1},b_{s+1})$ is
  strong, we simply set $(u_{s+1},v_{s+1})=(a_{s+1},b_{s+1})$ and continue.

  Now suppose the directed edge from $(a_s,b_s)$ to $(a_{s+1},b_{s+1})$ is weak
  and is evidenced by the strict alternating cycle $((a_s,b_s),(a_{s+1},b_{s+1}),
  (u,v))$ in $P$ with $(a_s,a_{s+1})$ and $(a_{s+1},u)$ both being left pairs. 
  Then we set $(u_{s+1},v_{s+1})=(a_{s+1},b_{s+1})$, so that the
  first statement of our inductive hypothesis is satisfied. 
  However, to maintain the
  second statement, 
  for every $s\ge2$ we replace $(u_s,v_s)=(a_s,b_s)$ by 
  $(u_s,v_s)=(a_s,v)$. 
  The new pair belongs to $I_\theta$ by the
  strictness of the helper cycle and Proposition~\ref{pro:all-inc-are-inside}.

  To complete the proof, we need only to show that the two requirements are
  satisfied by the updated path.  The first requirement, i.e.,
  $(u_{s+1},v_{s+1})=(a_{s+1},b_{s+1})$ is satisfied trivially.
  For the second, we note that there is a strong directed edge
  from $(a_s,v)$ to $(a_{s+1},b_{s+1})$.  It suffices to show
  that, for every $s\ge2$, there is a strong edge from $(u_{s-1},v_{s-1})$ to
  $(a_s,v)$.

  We know $(u_{s-1},a_s)$ is a left pair. 
  We also know that 
  $a_s< v_{s-1}$ in $P$. 
  It suffices to show that $u_{s-1}< v$ in $P$. 
  Let $N$ be a separating path in $\cgB$ associated with
  the comparability $u_{s+1}< v$ in $P$.  Since $(u_{s-1},u_{s+1})$ is
  a left pair, Proposition~\ref{pro:force-sides} implies $u_{s-1}$ is
  left of $N$.  Since $(b_s,v)$ is a left pair, and $b_s\parallel u_{s+1}$ in $P$,
  Proposition~\ref{pro:force-sides} implies that $b_s$ is right of $N$.
  Since $u_{s-1}< b_s$ in $P$, Proposition~\ref{pro:force-comp} implies
  $u_{s-1}< v$ in $P$.  This observation completes the proof that the inductive
  construction can continue.
\end{proof}

We have now shown that for each $\theta\in\{0,1\}$ and each $m\ge1$, the
set $I_\theta(m)$ is reversible in $P$.  We have also shown that $I_\theta(m)=
\emptyset$ when $m>\se(P)$. 

Together, these statements complete
the proof of Theorem~\ref{thm:dim-boundedness}, i.e.\ the upper bound $\dim(P)\le 2\se(P)+2$ when $P$ is
a poset with a planar cover graph and a unique minimal element.

For future reference, we state the following corollary.
\begin{corollary}
Let $P$ be a poset with a planar cover graph and 
let $x_0$ be a unique minimal element in $P$.
Fix a planar drawing of the cover graph of $P$ with $x_0$ on the exterior face.  

For all integers $k\ge2$,
if $\dim(P)\geq 2k+1$, then
there is a shadow block $\cgB$ in $P$ containing the elements in $k$ incomparable pairs $(a_1,b_1),\ldots,(a_k,b_k)$ such that
\begin{enumerate}
  \item The elements $a_1,\ldots,a_k,b_1,\ldots,b_k$ induce a standard example $S_k$ in $P$.
  \item $a_i\parallel y_{\cgB}$ in $P$ and $b_i>y_{\cgB}$ in $P$, for all $i\in[k]$.
  \item $(a_i,a_j)$ is a left pair and $(b_i,b_j)$ is a left pair for all $i,j\in[k]$ with $i<j$.
\end{enumerate}
\end{corollary}

\section{Acknowledgments and AI usage}

The authors are grateful to both reviewers for their helpful comments and suggestions. In particular, the detailed recommendations in one report led to substantial improvements in its structure and exposition. The first version of this paper was prepared in June 2022, and a revision was completed in October 2026. 

During the revision, we additionally used OpenAI's GPT-6 Astra to help identify and correct minor issues. It also suggested the formulation of the Completion lemma and the final argument in Section~\ref{sec:bdim} which reduced our bound on Boolean dimension from $13$ to $10$.

\bibliographystyle{abbrv}
\bibliography{boolean-dimension}

@article{BHKPTW20,
	author={Biró, Csaba and Hamburger, Peter and Kierstead, Hal A. and Pór, Atilla and Trotter, William T. and Wang, Ruidong},
	title={Random bipartite posets and extremal problems},
	journal={Acta Mathematica Hungarica},
	year={2020},
	volume={161},
	pages={618-646},
	note={\href{https://arxiv.org/abs/2003.07935}{arXiv:2003.07935}},
}

@article{BHPT16,
	author={Biró, Csaba and Hamburger, Peter and Pór, Atilla and Trotter, William T.},
	title={Forcing posets with large dimension to contain large standard examples},
	journal={Graphs and Combinatorics},
	volume={32},
	pages={861-880},
	year={2016},
	note={\href{https://arxiv.org/abs/1402.5113}{arXiv:1402.5113}},
}

@article{BHP15,
	author={Biró, Csaba and Hamburger, Peter and Pór, Atilla},
	title={Standard examples as subposets of posets},
	journal={Order},
	volume={32},
	pages={293-299},
	year={2015},
	note={\href{https://arxiv.org/abs/1311.6518}{arXiv:1311.6518}},
}

@article{Kel81,
title = {On the dimension of partially ordered sets},
journal = {Discrete Mathematics},
volume = {35},
number = {1},
pages = {135-156},
year = {1981},
note = {Special Volume on Ordered Sets},
issn = {0012-365X},
author = {David Kelly},
}

@unpublished{KMT21+,
author={Kozik, Jakub and  Micek, Piotr and Trotter, William T.},
title={Dimension is polynomial in height for posets with planar cover graphs},
note={\href{https://arxiv.org/abs/1907.00380}{arXiv:1907.00380}},
year={2019},
}

@incollection {NP89,
	AUTHOR = {Ne\v{s}et\v{r}il, J. and Pudl\'ak, P.},
	TITLE = {A note on {B}oolean dimension of posets},
	BOOKTITLE = {Irregularities of partitions ({F}ert\H od, 1986)},
	SERIES = {Algorithms Combin. Study Res. Texts},
	VOLUME = {8},
	PAGES = {137-140},
	PUBLISHER = {Springer, Berlin},
	YEAR = {1989},
}

@article{SS20,
author = {Scott, Alex and Seymour, Paul},
title = {A survey of $\chi$-boundedness},
journal = {Journal of Graph Theory},
volume = {95},
number = {3},
pages = {473-504},
year = {2020},
note = {\href{https://arxiv.org/abs/1812.07500}{arXiv:1812.07500}}
}

@inproceedings{Tro78,
	author={William T.\ Trotter},
	title={Order preserving embeddings of aographs},
	booktitle={Theory and Applications of Graphs},
	publisher={Springer-Verlag},
	volume={642},
	year={1978},
	pages={572-579},
}

@book{Tro-book,
	Address = {Baltimore, MD},
	Author = {Trotter, William T.},
	Isbn = {0-8018-4425-8},
	Mrclass = {06-02 (05-02 06A06 06A07)},
	Mrnumber = {1169299 (94a:06001)},
	Mrreviewer = {K. P. Bogart},
	Pages = {xvi+307},
	Publisher = {Johns Hopkins University Press},
	Series = {Johns Hopkins Series in the Mathematical Sciences},
	Title = {Combinatorics and partially ordered sets: Dimension theory},
	Year = {1992},
}

@article{TM77,
	Author = {Trotter, Jr., William T. and Moore, Jr., John I.},
	Journal = {Journal of Combinatorial Theory, Series B},
	Mrclass = {06A10},
	Mrnumber = {0434893 (55 \#7857)},
	Mrreviewer = {K. P. Bogart},
	Number = {1},
	Pages = {54--67},
	Title = {The dimension of planar posets},
	Volume = {22},
	Year = {1977}}

@inproceedings{BEGS21,
author = {Bonamy, Marthe and Esperet, Louis and Groenland, Carla and Scott, Alex},
title = {Optimal Labelling Schemes for Adjacency, Comparability, and Reachability},
year = {2021},
publisher = {Association for Computing Machinery},
booktitle = {Proceedings of the 53rd Annual ACM SIGACT Symposium on Theory of Computing},
pages = {1109–1117},
numpages = {9},
series = {STOC 2021},
note = {\href{https://arxiv.org/abs/2012.01764}{arXiv:2012.01764}}
}

@INPROCEEDINGS{HRT15,
  author={Holm, Jacob and Rotenberg, Eva and Thorup, Mikkel},
  booktitle={2015 IEEE 56th Annual Symposium on Foundations of Computer Science}, 
  title={Planar Reachability in Linear Space and Constant Time}, 
  year={2015},
  volume={},
  number={},
  pages={370-389},
	note = {\href{https://arxiv.org/abs/1411.5867}{arXiv:1411.5867}}
 }

@article{Tho04,
author = {Thorup, Mikkel},
title = {Compact Oracles for Reachability and Approximate Distances in Planar Digraphs},
year = {2004},
publisher = {Association for Computing Machinery},
address = {New York, NY, USA},
volume = {51},
number = {6},
journal = {Journal of the ACM},
pages = {993–1024},
numpages = {32},
}

@article{GNT90,
title = "On locally presented posets",
journal = "Theoretical Computer Science",
volume = "70",
number = "2",
pages = "251 - 260",
year = "1990",
issn = "0304-3975",
doi = "https://doi.org/10.1016/0304-3975(90)90125-2",
url = "http://www.sciencedirect.com/science/article/pii/0304397590901252",
author = "Giorgio Gambosi and Jaroslav Nešetřil and Maurizio Talamo",
}

@article{DM41,
	Author = {Dushnik, Ben and Miller, Edwin W.},
	Fjournal = {American Journal of Mathematics},
	Issn = {0002-9327},
	Journal = {Amer. J. Math.},
	Mrclass = {27.2X},
	Mrnumber = {0004862 (3,73a)},
	Mrreviewer = {G. Birkhoff},
	Pages = {600--610},
	Title = {Partially ordered sets},
	Volume = {63},
	Year = {1941}
}

@article{FMM20,
author={Felsner, Stefan and Mészáros, Tamás and Micek, Piotr},
title={Boolean Dimension and Tree-Width},
journal={Combinatorica},
volume={40},
year={2020},
pages={655-677},
note={\href{https://arxiv.org/abs/1707.06114}{arXiv:1707.06114}},
}

@article{MMT19,
Author = {M\'{e}sz\'{a}ros, Tam\'{a}s and Micek, Piotr and Trotter, William T.},
Title = {Boolean dimension, components and blocks},
journal = "Order",
year = "2019",
NOTE = {\href{https://arxiv.org/abs/1801.00288}{arXiv:1801.00288}},
}

@misc{BHMST25,
  author={Blake, Heather Smith and Hodor, J{\k e}drzej and Micek, Piotr and Seweryn, Micha{\l} T. and Trotter, William T.},
  title={Planarity and dimension {I}},
  year={2025},
  note={\href{https://arxiv.org/abs/2510.18603v2}{arXiv:2510.18603v2}}
}

@misc{BHMST26,
  author={Blake, Heather S. and Hodor, J{\k e}drzej and Micek, Piotr and Seweryn, Micha{\l} T. and Trotter, William T.},
  title={Planarity and dimension {II}},
  year={2026},
  note={\href{https://arxiv.org/abs/2607.09294}{arXiv:2607.09294}}
}
\end{document}